\documentclass[a4paper, 11pt]{article}

\usepackage[utf8]{inputenc}
\usepackage[english]{babel}
\usepackage{amsthm, amsmath, amssymb}
\usepackage[T1]{fontenc}
\usepackage{fullpage}
\usepackage{lipsum}
\usepackage{authblk}
\usepackage{url} 
\usepackage{hyperref}
\usepackage{graphicx}
\usepackage{mathtools}
\usepackage{bbm}
\usepackage{tikz-cd,tikz-3dplot}
\usepackage{enumitem}
\usepackage{soul}
\usepackage{mathrsfs}
\usepackage{subcaption}
\usepackage{float}
\usepackage{longtable}
\usepackage{array}

\usepackage[maxbibnames=99, style=numeric]{biblatex} 

\newtheorem{assumption}{Assumption}[section]
\newtheorem{definition}[assumption]{Definition}
\newtheorem{theorem}[assumption]{Theorem}
\newtheorem{proposition}[assumption]{Proposition}
\newtheorem{corollary}[assumption]{Corollary}
\newtheorem{lemma}[assumption]{Lemma}

\theoremstyle{remark}

\theoremstyle{remark}
\newtheorem{remark}[assumption]{Remark}

\newcommand{\Acal}{\mathcal{A}}

\newcommand{\Ccal}{\mathcal{C}}
\newcommand{\Dcal}{\mathcal{D}}
\newcommand{\Ecal}{\mathcal{E}}
\newcommand{\Fcal}{\mathcal{F}}
\newcommand{\Gcal}{\mathcal{G}}

\newcommand{\Ical}{\mathcal{I}}

\newcommand{\Lcal}{\mathcal{L}}
\newcommand{\Mcal}{\mathcal{M}}
\newcommand{\Ncal}{\mathcal{N}}

\newcommand{\Pcal}{\mathcal{P}}

\newcommand{\Scal}{\mathcal{S}}

\newcommand{\Vcal}{\mathcal{V}}
\newcommand{\Wcal}{\mathcal{W}}

\newcommand{\Prob}{\mathscr{P}}
\newcommand{\Mart}{\mathscr{M}}

\newcommand{\EE}{\mathbb{E}}

\newcommand{\NN}{\mathbb{N}}

\newcommand{\PP}{\mathbb{P}}

\newcommand{\RR}{\mathbb{R}}

\newcommand{\mean}{{\textup{mean}}}
\newcommand{\diam}{{\textup{diam}}}

\newcommand{\CDF}{{\textup{CDF}}}

\newcommand{\supp}{{\textup{supp}}}
\newcommand{\co}{{\textup{co}}}

\newcommand{\Mon}{{\operatorname{Mon}_{\uparrow}}}
\newcommand{\Law}{{\textup{Law}}}
\newcommand{\interior}{{\textup{int}}}
\newcommand{\MCov}{{\textup{MCov}}}

\numberwithin{equation}{section}

\makeatletter
\renewcommand{\p@enumii}{}            

\makeatother

\newcounter{stepimg}

\newcommand{\stepfig}[2]{%
  \refstepcounter{figure}\label{#1}\setcounter{stepimg}{0}%
  \begin{longtable}{|@{\hspace{.5em}}
  >{\centering\arraybackslash}p{.32\textwidth}@{\hfill}
  >{\centering\arraybackslash}p{.32\textwidth}@{\hfill}
  >{\centering\arraybackslash}p{.32\textwidth}
  @{\hspace{1em}}|}
  \hline
  \multicolumn{3}{|@{\hspace{.5em}}l@{\hspace{.5em}}|}{%
    \rule{0pt}{1.5em}\textbf{Figure \thefigure: #2}%
  }\\[0.5em]
  \hline
  \noalign{\vskip .25em}
}

\newcommand{\closestepfig}{%
  \hline
  \end{longtable}%
}

\begin{document}

\title{Fixed Points for the $q$-Bass Martingale:\\ Existence, Stability, and Convergence}
\author{Beatrice Acciaio\thanks{Department of Mathematics, ETH Z\"{u}rich, Switzerland. \emph{beatrice.acciaio@math.ethz.ch}}~~and~Antonio Marini\thanks{Department of Mathematics, ETH Z\"{u}rich, Switzerland. \emph{antonio.marini@math.ethz.ch}}}

\maketitle

\begin{abstract} We establish existence, uniqueness, stability, and convergence results for one-dimensional $q$-Bass martingales, characterized as the martingales with prescribed initial and terminal marginals whose transition kernels are closest to a reference measure $q$. Their existence is equivalent to the solvability of a fixed-point problem for probability distributions. Building on \cite{AcMa26_semidiscrete}, that requires the first marginal to be supported on finitely many points, we study the case of general marginals in convex order. Under the assumption that $q\ll\lambda$, we prove existence, uniqueness and stability of fixed-point distributions, $\Wcal_\infty$-convergence of the fixed-point iteration, and support-diameter estimates. We also extend the martingale Benamou--Brenier formula from Brownian motion to any additive reference process $X$ and show that the corresponding $X$-Bass martingale is optimal whenever it exists,  with an interpretation as an adapted Wasserstein projection of $X$.
 \end{abstract}

\section{Introduction}

For any probability measures  $\mu$ and $\nu$ in convex order, Strassen's theorem guarantees the existence of martingale couplings between them, though it does not provide an explicit construction, nor does it identify a particular coupling among the admissible ones.
The idea behind the $q$-Bass martingale studied in this paper
goes back to Bass' solution to the Skorokhod embedding problem \cite{Ba83}. 
Given a target distribution $\nu$, Bass' construction consists in transporting a Gaussian random variable monotonically to $\nu$ and then taking conditional expectations to obtain a martingale starting from a deterministic initial value and with the prescribed terminal distribution.
A suitable time change finally turns this martingale into a solution of the embedding problem.
This mechanism was subsequently adapted to the setting where the initial marginal $\mu$ is any distribution in convex order with the terminal one (rather than a Dirac measure) through the concept of stretched Brownian motion \cite{BaBeHuKa20}, where the Gaussian law serves as a reference for selecting the martingale dynamics. Replacing this Gaussian reference by an arbitrary probability measure $q$ leads to the broader framework considered in the present paper.

The resulting $q$-Bass construction selects, among martingales with prescribed initial and terminal marginals, one whose transition kernels deviate as little as possible from the reference measure $q$. In one dimension, the construction admits an explicit description in terms of monotone transport. Given a probability distribution $\alpha$, one first transports the convolution measure $\alpha\ast q$ monotonically to the target measure $\nu$ and then averages the corresponding transport map $T$ against $q$, obtaining the map $x\mapsto \int_\RR T(x+z)\,q(dz)$. If this averaged map pushes $\alpha$ forward to  the initial marginal $\mu$, the construction yields a martingale from $\mu$ to $\nu$. We call $\alpha$ a fixed-point distribution, or Bass distribution.
The importance of this construction lies in its optimal transport characterization. When the marginals have finite second moments, the $q$-Bass martingale is the unique optimizer of the weak martingale transport problem introduced in \cite{Ts24}. Equivalently, it minimizes, on average with respect to the initial marginal, the squared Wasserstein distance between its transition kernels and the reference measure $q$. The problem also admits a dual formulation in terms of convex potentials.
For a Gaussian reference measure, Hasenbichler--Joseph--Loeper--Obloj--Pammer established in \cite{HaJoLoObPa26} the existence of an optimizer when the prescribed marginals have finite $p$-th moments for some $p>1$. 

In the one-dimensional setting, the condition that the transport map averaged against $q$ pushes the Bass distribution forward to $\mu$ can be formulated as a fixed-point equation. This approach was initiated by Conze and Henry-Labordère \cite{CoHe21}, who introduced a fixed-point equation for the cumulative distribution function of the Bass distribution for computational purposes in the calibration of the Bass Local Volatility model. A formulation in terms of quantile functions was subsequently developed in \cite{AcMaPa23}. This quantile-based equation was then extended to arbitrary reference measures $q$ in \cite[Theorem~1.4]{AcMa26_semidiscrete}, which shows that a $q$-Bass martingale from $\mu$ to $\nu$ exists if and only if there is a distribution $\alpha$ whose quantile function satisfies the associated fixed-point equation, together with the required monotonicity condition. While this characterization is not restricted to discrete marginals, the existence theory developed in \cite{AcMa26_semidiscrete} concerns the case in which the initial marginal is supported on finitely many atoms. The present paper removes this restriction and develops the corresponding general theory. The formal definition of a $q$-Bass martingale, its variational characterization, and the fixed-point formulations are given in Section~\ref{sec:q-bass-construction}.

We now present the main contributions of the current paper. The first one concerns the existence of $q$-Bass martingales beyond the semidiscrete setting considered in \cite{AcMa26_semidiscrete}. More precisely, our first theorem removes the assumption in \cite[Theorem~1.9]{AcMa26_semidiscrete} that the initial marginal be supported on finitely many atoms. It also covers a class of non-irreducible pairs characterized by finitely many contact points between their integrated quantile functions.

\begin{theorem}[Existence of the $q$-Bass martingale]
\label{thm:q-bass-existence}
	Let $\mu,\nu\in\Prob_1(\RR)$ and let $q\in\Prob(\RR)$ satisfy $q\ll\lambda$.
    Assume that one of the following conditions holds:
	\begin{enumerate}[label=(\roman*)]
		\item \label{it:existence-1} $(\mu,\nu)$ is irreducible;
		\item \label{it:existence-2} $\mu\preceq_c\nu$,  $q$ is compactly supported, and the set $K:=\{u\in(0,1):U_\mu(u)=U_\nu(u)\}$ satisfies
		\begin{equation} \label{eq:compact-support-components}
			|K|<\infty, \qquad \{u\in K:Q_\nu(u^+)=Q_\mu(u^+)\text{ or }Q_\nu(u)=Q_\mu(u)\}=\emptyset.
		\end{equation}
	\end{enumerate}
	Then there exists a $q$-Bass martingale from $\mu$ to $\nu$.
\end{theorem}

The proof is given in Section~\ref{sec:existence-uniqueness}. It builds on the semidiscrete existence result of \cite{AcMa26_semidiscrete} by approximating the initial marginal with measures supported on finitely many atoms and applying a generalization of Helly's selection theorem to pass to the limit. In particular, no regularity of the marginals beyond finite first moments is required, and absolute continuity is the only regularity imposed on the reference measure.

The uniqueness and stability results are obtained under the conditions collected in Assumptions~\ref{ass:technical-assumptions} and \ref{ass:technical-assumptions-2}, stated in Section~\ref{sec:q-bass-construction}. The role of Assumption~\ref{ass:technical-assumptions-2} is clarified by the geometric analysis developed in Section~\ref{sec:support}. There, we characterize when a fixed-point distribution is compactly supported and introduce the $r$-maximal curves (see Definition~\ref{def:r-maximal-curve}), which provide lower bounds on the diameter of its support in terms of $\mu$, $\nu$, and $q$. In particular, Assumption~\ref{ass:technical-assumptions-2} guarantees that every fixed-point distribution is compactly supported, a property used in the uniqueness, stability, and convergence results.

The next result extends the uniqueness shown in \cite[Theorem~1.12]{AcMa26_semidiscrete} from finitely supported initial marginals to the general setting. Its proof is given in Section~\ref{sec:existence-uniqueness}.

\begin{theorem}[Uniqueness of the fixed-point solution]
\label{thm:uniqueness-fixed-point-sol}
Let $\mu,\nu\in\Prob_1(\RR)$ and $q\in\Prob(\RR)$ be such that Assumptions~\ref{ass:technical-assumptions} and \ref{ass:technical-assumptions-2} hold and $(\mu,\nu)$ is irreducible. Then the fixed-point solution is unique up to translation.
\end{theorem}

 The stability result below complements the finite-dimensional stability of the atomic $q$-Bass maps established in \cite[Proposition~2.12]{AcMa26_semidiscrete}. Its proof is given in Section~\ref{sec:stability}.

\begin{theorem}[Stability of the fixed-point solution]
\label{thm:fixed-point-stability}
Let $\mu,\nu\in\Prob_1(\RR)$ and $q\in\Prob(\RR)$ be such that Assumptions~\ref{ass:technical-assumptions} and \ref{ass:technical-assumptions-2} hold and $(\mu,\nu)$ is irreducible. Let $(\mu_k)_{k\in\NN},(\nu_k)_{k\in\NN}\subseteq\Prob_1(\RR)$ and $(q_k)_{k\in\NN}\subseteq\Prob(\RR)$ satisfy
\[
\mu_k\preceq_c\nu_k \qquad\text{and}\qquad q_k\ll\lambda, \qquad \text{for all } k\in\NN.
\]
Assume moreover that
\[
U_{\mu_k}\to U_\mu \quad\text{pointwise on }(0,1),
\]
\[
Q_{\nu_k}\to Q_\nu \quad\text{pointwise on }(0,1),
\]
that there exists $H\in L^1(0,1)$ such that $|Q_{\nu_k}|\leq H$ on $(0,1)$, for all $k\in\NN$, and that
\[
\rho_{q_k}\to\rho_q \quad\text{in }L^1(\RR).
\]
Finally, for each $k\in\NN$, assume that there exists a fixed-point distribution $\alpha_k\in\Prob(\RR)$ associated with a $q_k$-Bass martingale from $\mu_k$ to $\nu_k$, normalized by $Q_{\alpha_k}(1/2)=0$. Then $(\alpha_k)_{k\in\NN}$ converges weakly to the fixed-point distribution $\alpha$ associated with the $q$-Bass martingale from $\mu$ to $\nu$ and normalized by $Q_\alpha(1/2)=0$.
\end{theorem}

In Section~\ref{sec:convergence}, we combine the $\Wcal_\infty$-non-expansiveness of the fixed-point operator, established in \cite{AcMaPa23} through a detailed analysis of its regularity, with the dual-descent argument of \cite{HaJoLoObPa26}. This yields $\Wcal_\infty$-convergence of the fixed-point iteration to a fixed-point distribution in the compact setting; under additional assumptions, we further prove linear convergence of the fixed-point algorithm. In particular, we obtain the following result; see Section~\ref{sec:convergence} for the proof.

\begin{theorem}[Convergence of the fixed-point iteration]
\label{thm:convergence}
Let $\mu,\nu\in\Prob_\infty(\RR)$ and $q\in\Prob_1(\RR)$ satisfy $q\ll\lambda$ and Assumption~\ref{ass:technical-assumptions-2}. Denote by $\Lcal^\ast$ the set of fixed-point distributions and assume that $\Lcal^\ast\neq\emptyset$. Let $\alpha_0\in\Prob_\infty(\RR)$, denote its quantile function by $Q_0$, and define recursively
\[
Q_{k+1}:=\Gcal_qQ_k,\qquad k\in\NN,
\]
where $\Gcal_q$ is the operator defined in Definition~\ref{def:G_q}. For each $k\in\NN$, let $\alpha_k$ be the probability measure with quantile function $Q_k$. Then, for every $p\in[1,\infty)$,
\[
\Wcal_p(\alpha_k,\Lcal^\ast)\longrightarrow0,
\]
where $\Wcal_s(\eta,\Lcal^\ast):=\inf_{\bar\alpha\in\Lcal^\ast}\Wcal_s(\eta,\bar\alpha)$ for $s\in[1,\infty]$ and $\eta\in\Prob_\infty(\RR)$.

Assume, in addition, that one of the following conditions hold:
\begin{enumerate}[label=(\roman*)]
    \item \label{it:convergence-assumption-1} $q$ is equivalent to the Lebesgue measure;
    \item \label{it:convergence-assumption-2} the supports of $q$, $\mu$, and $\alpha_0$ are intervals, and $\mu$ and $\nu$ are absolutely continuous.
\end{enumerate}
Then
\[
\Wcal_\infty(\alpha_k,\Lcal^\ast)\longrightarrow0.
\]

Moreover, if Assumption~\ref{ass:linear-convergence} holds, then there exist $\theta\in(0,1)$ and $\widetilde\alpha\in\Lcal^\ast$ such that
\begin{enumerate}[label=(\alph*)]
    \item \label{it:strong-conv-1} for every $k\in\NN$,
    \[
    \Wcal_\infty(\alpha_{k+1},\Lcal^\ast)\leq\theta\,\Wcal_\infty(\alpha_k,\Lcal^\ast);
    \]
    \item \label{it:strong-conv-2} for every $k\in\NN$,
    \[
    \Wcal_\infty(\alpha_{k+1},\widetilde\alpha)\leq2\theta^k\Wcal_\infty(\alpha_1,\widetilde\alpha).
    \]
\end{enumerate}
\end{theorem}

Finally, in Section~\ref{sec:dynamic}, we investigate a continuous-time counterpart of the $q$-Bass construction, in which only the initial and terminal marginals are prescribed. Given an additive reference process $X$, we introduce the notion of an $X$-Bass martingale and observe that, if $X_1\sim q$, then, for every $X$-Bass martingale $(M_t)_{t\in[0,1]}$ from $\mu$ to $\nu$, the joint law $\Law(M_0,M_1)$ is a $q$-Bass martingale from $\mu$ to $\nu$. Using the martingale representation in the filtration generated by $X$, we derive dynamic formulations of the optimization problems \eqref{def:WT^q_S} and \eqref{def:WT^q_I}. These formulations compare separately the continuous and jump characteristics of an admissible martingale with those of the reference process. We then prove that, whenever it exists, the $X$-Bass martingale solves the corresponding dynamic martingale optimal transport problem. In Section~\ref{sec:aw-projection}, we complement this characterization by showing that its law is the adapted Wasserstein projection of $\Law(X)$ onto the set of martingale laws with prescribed initial and terminal marginals.

\paragraph{Notation.}
We denote by $\Prob(\RR)$ the set of probability measures on $\RR$, by $\Prob_p(\RR)$, $p\in[1,\infty)$, the subset of measures with finite $p$-moment, and by $\Prob_\infty(\RR)$ the subset of compactly supported probability measures. The standard Gaussian law and its density are denoted by $\gamma$ and $\phi$, respectively, while $\lambda$ denotes the Lebesgue measure on $\RR$. For $\xi\in\Prob(\RR)$, we write $\xi\ll\lambda$ when $\xi$ is absolutely continuous with respect to $\lambda$, and in this case denote its density by $\rho_\xi$. We write $\supp(\xi)$ and $F_\xi$ for the support and cumulative distribution function of $\xi$, respectively. Its left-continuous quantile function is denoted by $Q_\xi:(0,1)\to\RR$ and defined by $Q_\xi(u):=\inf\{y\in\RR:F_\xi(y)\geq u\}$. We extend it to $[0,1]$ by setting $Q_\xi(0):=\inf(\supp(\xi))$ and $Q_\xi(1):=\sup(\supp(\xi))$. If $\xi\in\Prob_1(\RR)$, we set $\mean(\xi):=\int_\RR y\,\xi(dy)$. We use $\CDF$ both as an abbreviation for cumulative distribution function and to denote the class of all cumulative distribution functions on $\RR$. We denote by $\CDF_c \subseteq \CDF$ the subclass corresponding to compactly supported probability measures.

For $\mu,\nu\in\Prob(\RR)$, we denote by $\Pi(\mu,\nu)$ the set of probability measures on $\RR\times\RR$ with first marginal $\mu$ and second marginal $\nu$, and its elements are called couplings of $\mu$ and $\nu$. For $p\in[1,\infty)$ and $\xi,\zeta\in\Prob_p(\RR)$, the $p$-Wasserstein distance between $\xi$ and $\zeta$ is defined as
\[
\Wcal_p(\xi,\zeta):=\inf_{\pi\in\Pi(\xi,\zeta)}\left(\int_{\RR\times\RR}|x-y|^p\,\pi(dx,dy)\right)^{1/p}.
\]
Given $\pi\in\Pi(\mu,\nu)$, we write $\pi(dx,dy)=\mu(dx)\pi_x(dy)$ for its regular conditional disintegration with respect to $\mu$. We denote by $\Mart(\mu,\nu)$ the subset of couplings satisfying $\mean(\pi_x)=x$ for $\mu$-a.e.\ $x$, and call its elements martingale couplings of $\mu$ and $\nu$.

For a measurable map $T:\RR\to\RR$ and $\xi\in\Prob(\RR)$, the push-forward of $\xi$ under $T$ is denoted by $T_\#\xi$ and defined by $T_\#\xi(A):=\xi(T^{-1}(A))$ for every Borel set $A\subseteq\RR$. For $\xi,\zeta\in\Prob(\RR)$, their convolution is the probability measure $\xi\ast\zeta$ given by $(\xi\ast\zeta)(A):=\int_{\RR\times\RR}1_A(x+y)\,\xi(dx)\zeta(dy)$. For measurable functions $f,g$, we set $f\ast g(x):=\int_\RR f(x-y)g(y)\,dy$, and for a measurable function $f$ and $\xi\in\Prob(\RR)$ we define $\xi \ast f(x):=\int_\RR f(x-y)\,\xi(dy)$ and $\xi \star f(x):=\int_\RR f(x+y)\,\xi(dy)$, whenever these quantities are well-defined.

For $A\subseteq\RR$ and $p\in[1,\infty)$, $L^p(A)$ denotes the space of $\lambda$-measurable functions $f:A\to\RR$ such that $|f|^p$ is $\lambda$-integrable, while $L^\infty(A)$ denotes the space of $\lambda$-essentially bounded functions. For $a,b\in\RR$ and $p\in[1,\infty]$, we use the shorthand $L^p(a,b):=L^p((a,b))$, and analogously for $(a,b]$, $[a,b)$, and $[a,b]$. The usual $L^p$ norm is denoted by $\|\cdot\|_p$, with $\|\cdot\|_\infty$ denoting the essential supremum norm.

For $X \subseteq \mathbb \overline{\RR}$, we denote by $\Mon((0,1),X)$ the set of non-decreasing, left-continuous functions from $(0,1)$ to $X$. For a function $v:\RR \to \RR$, we denote by $v^*(y):= \sup_{x \in \RR} (xy - v(x))$ its convex conjugate.

\section{The $q$-Bass construction and its fixed-point formulation}
\label{sec:q-bass-construction}

This section collects the formal definitions and structural properties underlying the results stated in the introduction. We first recall the concept of convex order and its description through integrated quantile functions, then introduce $q$-Bass martingales and their variational interpretation, and finally present the associated fixed-point formulation and the standing assumptions used throughout the paper.

\subsection{Convex order and integrated quantiles}\label{sec:cvx-order}

We recall the notions used to formulate the existence results.

\begin{definition}[Convex order and irreducibility]
\label{def:cvx-order-irreducibility}
Let $\mu,\nu\in\Prob_1(\RR)$. We say that $\mu$ is dominated by $\nu$ in convex order, and write $\mu\preceq_c\nu$, if
\[
\int_\RR\varphi\,d\mu\leq\int_\RR\varphi\,d\nu
\]
for every convex function $\varphi:\RR\to\RR$ for which both integrals are well defined. For $\rho\in\Prob_1(\RR)$, define
\[
C_\rho(x):=\int_\RR(y-x)^+\,\rho(dy),\qquad x\in\RR.
\]
The pair $(\mu,\nu)$ is called irreducible if $\mu\preceq_c\nu$ and the set
\[
I_{\mu,\nu}:=\{x\in\RR:C_\mu(x)<C_\nu(x)\}
\]
is an interval satisfying $\mu(I_{\mu,\nu})=1$.
\end{definition}

The potential characterization of convex order states that $\mu\preceq_c\nu$ if and only if $\mu$ and $\nu$ have the same mean and $C_\mu\leq C_\nu$ on $\RR$; see \cite[Section~2.2]{BeJu16}.

\begin{definition}[Integrated quantile function]
\label{def:integ_quant}
For $\eta\in\Prob_1(\RR)$, its integrated quantile function is the map $U_\eta:[0,1]\to\RR$ defined by
\[
U_\eta(p):=\int_0^pQ_\eta(u)\,du.
\]
\end{definition}

We will repeatedly use the characterization proved in \cite[Proposition~1.8]{AcMa26_semidiscrete}. Namely, $U_\eta$ is convex, $U_\eta(0)=0$, and $U_\eta(1)=\mean(\eta)$. Moreover, for $\mu,\nu\in\Prob_1(\RR)$,
\[
\mu\preceq_c\nu \quad\Longleftrightarrow\quad U_\mu(1)=U_{\nu}(1)\ \text{ and }\ U_\mu(p)\geq U_{\nu}(p)\ \text{ for every }p\in(0,1),
\]
while $(\mu,\nu)$ is irreducible precisely when the last inequality is strict throughout $(0,1)$.

\begin{remark}[Irreducible components]
\label{rmk:irreducible-components}
When the pair $(\mu,\nu)$ is not irreducible, the open set $I_{\mu,\nu}$ decomposes into pairwise disjoint connected components. The canonical decomposition of \cite[Section~2.2]{BeJu16} associates with each of these intervals an irreducible pair of subprobability measures, while the remaining common part of $\mu$ and $\nu$ is transported by the identity. After normalization, these pairs are referred to as the \emph{irreducible components} of $(\mu,\nu)$.

The decomposition can often be visualized through the contact points of $U_\mu$ and $U_\nu$. In particular, when the curves meet at finitely many points and the inequality $U_\mu>U_\nu$ is strict between consecutive contact points, these points divide the total mass into blocks corresponding to the irreducible components. In full generality, this interpretation must be understood through the canonical decomposition above, as the contact set may have a more complicated structure.
For instance, in Figure~\ref{fig:integrated_quantile_2}, we take
\[
\mu=\operatorname{Unif}_{[-2,-1]\cup[1,2]},\qquad \nu=\operatorname{Unif}_{[-3,3]}.
\]
The integrated quantile functions meet at $p=1/2$ and satisfy $U_\mu>U_\nu$ on both $(0,1/2)$ and $(1/2,1)$. The corresponding irreducible intervals are $(-3,0)$ and $(0,3)$, and the associated normalized irreducible pairs are
\[
\bigl(\operatorname{Unif}_{[-2,-1]},\operatorname{Unif}_{[-3,0]}\bigr)\qquad\text{and}\qquad\bigl(\operatorname{Unif}_{[1,2]},\operatorname{Unif}_{[0,3]}\bigr).
\]
\end{remark}

\begin{figure}[H]
\centering
\includegraphics[width=0.5\textwidth]{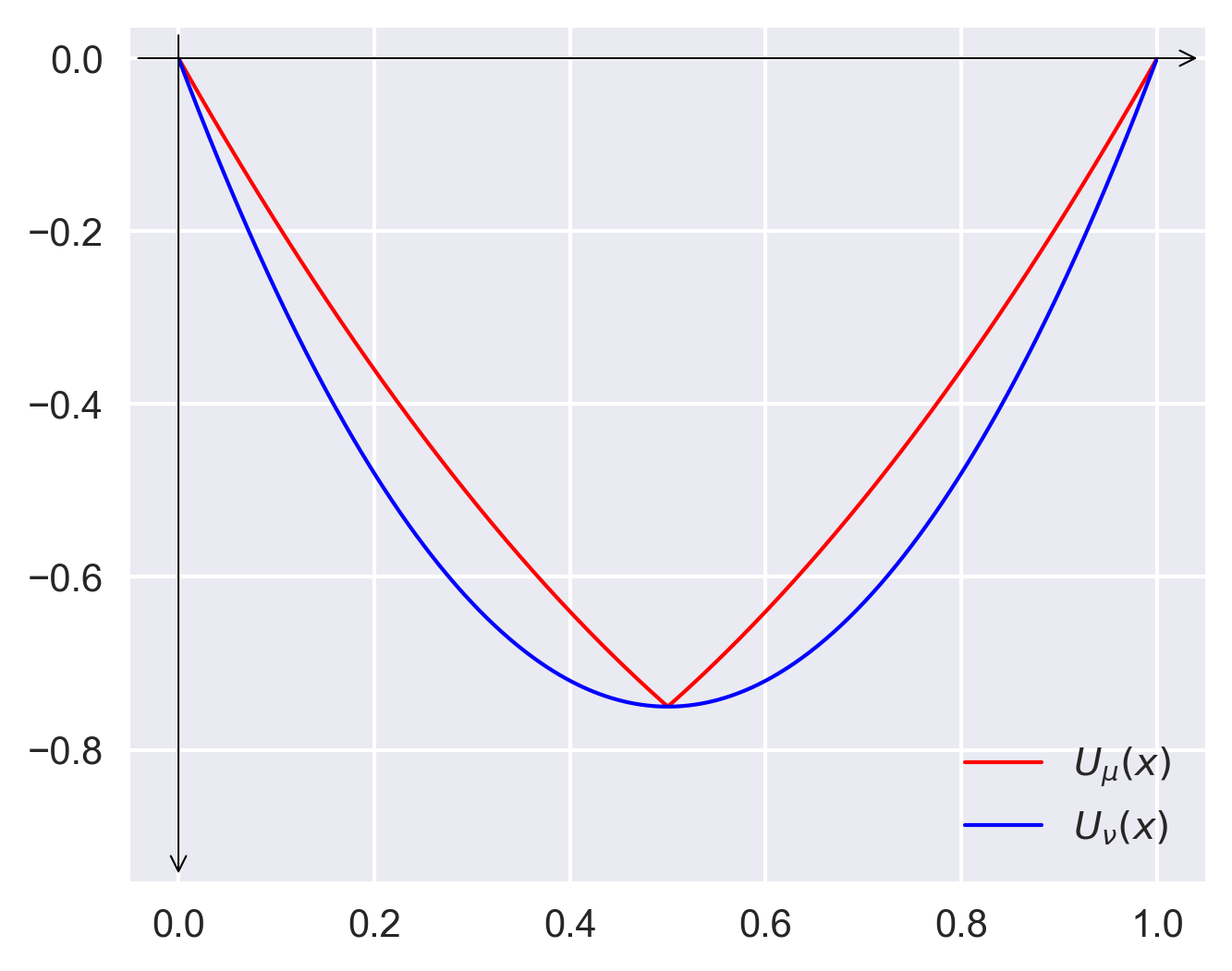}
\caption{Integrated quantile functions of $\mu=\operatorname{Unif}_{[-2,-1]\cup[1,2]}$ and $\nu=\operatorname{Unif}_{[-3,3]}$.
}
\label{fig:integrated_quantile_2}
\end{figure}

\subsection{The $q$-Bass martingale}

The notion of $q$-Bass martingale was introduced by Tschiderer in \cite{Ts24} through a formulation based on dual convex potentials. In one dimension, monotone transport maps admit an explicit representation in terms of cumulative distribution and quantile functions, which yields the equivalent definition given below. 

\begin{definition}[$q$-Bass martingale]
\label{def:q-Bass}
	Let $\mu, \nu \in \Prob_1(\RR)$, let $q \in \Prob(\RR)$ satisfy $q\ll \lambda$, and let
	\[
		(M_0,M_1)\sim\pi\in\Mart(\mu,\nu).
	\]
	We say that $\pi$ is a $q$-Bass martingale if there exists a measure $\alpha \in \Prob(\RR)$ such that
	\begin{equation}
	\label{eq:q-Bass-representation}
		M_0= (q \star T)(X_0), \qquad M_1=T(X_0+Q), \qquad \text{and} \qquad q \star T \text{ is increasing $\alpha$-a.s.},
	\end{equation}
	where $X_0 \sim \alpha$, $Q\sim q$, $X_0$ is independent of $Q$, and $T=Q_\nu \circ F_{\alpha \ast q}$.
\end{definition}

More specifically, any $q$-Bass martingale from $\mu$ to $\nu$ makes the following diagram commute.
\[
\begin{tikzcd}[row sep=0.7in, column sep = 1.4in]
  M_0 \sim \mu \arrow[r, "\text{\normalsize $q$-Bass martingale}"]  & M_1 \sim \nu  \\
  X_0 \arrow[u, "\text{\normalsize $q \star T$}"] \sim \alpha \arrow[r] & X_0+Q \sim \alpha \ast q  \arrow[u, "\text{\normalsize $T$}"]
\end{tikzcd}
\]

The measure $\alpha$ in Definition~\ref{def:q-Bass} will be referred to as either fixed-point distribution or Bass distribution. The monotonicity requirement in \eqref{eq:q-Bass-representation} corresponds to the strict convexity condition on the potential $q\star\hat v$ in \cite[Definition~1.4]{Ts24}, with $T=\nabla\hat v$ in the present notation. The map $q\star T$ is $\alpha$-a.e. well-defined under the assumptions of the definition by \cite[Proposition~1.3]{AcMa26_semidiscrete}; its monotonicity, however, genuinely depends on the choice of $\alpha$, as discussed in \cite[Remark~1.2]{AcMa26_semidiscrete}. Finally, the existence of such a martingale necessarily implies $\mu\preceq_c\nu$, since Strassen's theorem \cite{St65} characterizes the non-emptiness of $\Mart(\mu,\nu)$ by convex order.

\subsection{Variational interpretation}

The relevance of the above construction comes from its optimality in martingale transport. If $\mu,\nu,q\in\Prob_2(\RR)$, Tschiderer \cite{Ts24} showed that a $q$-Bass martingale from $\mu$ to $\nu$ is the unique optimizer of
\begin{equation}
	\label{def:WT^q_S}
	WT^q_\Scal(\mu,\nu)=\sup_{\pi\in\Mart(\mu,\nu)}\Scal(\pi), \qquad \Scal(\pi):=\int_\RR\mu(dx)\MCov(\pi_x, q),
\end{equation}
where, for any $\eta, \zeta \in \Prob(\RR)$, 
\[
\MCov(\eta, \zeta):=\sup_{p\in\Pi(\eta,\zeta)}\int_{\RR^2}xy\,p(dx,dy)
\]
denotes the maximal covariance between $\eta$ and $\zeta$.
Equivalently, up to terms depending only on the marginals and on $q$, it minimizes the average squared Wasserstein distance of the transition kernels from the reference measure:
\begin{equation}
	\label{def:WT^q_I}
    WT^q_\Ical(\mu,\nu)=\inf_{\pi\in\Mart(\mu,\nu)}\Ical(\pi), \qquad \Ical(\pi):=\int_\RR\mu(dx)\,\Wcal_2^2(\pi_x,q),
\end{equation}
where $\Wcal_2$ denotes the $2$-Wasserstein distance. Under the same assumptions, \eqref{def:WT^q_S} admits the dual representation
\begin{equation}
\label{eq:dual-problem}
	D^q(\mu,\nu)=\inf_{\substack{\psi\in L^1(\nu),\\ \psi\text{ convex}}}\left(\int_\RR\psi(x)\,\nu(dx)-\int_\RR(\psi^*\star q)^*(x)\,\mu(dx)\right).
\end{equation}

\subsection{The fixed-point equation}
In dimension one, the condition $(q\star T)_\#\alpha=\mu$ arising from Definition~\ref{def:q-Bass} can be expressed through equations for the distribution function or the quantile function of $\alpha$. The formulation in terms of cumulative distribution functions was introduced by Conze and Henry-Labordère \cite{CoHe21} for computational purposes in the calibration of the Bass Local Volatility model, while a quantile-based formulation was developed in \cite{AcMaPa23}. Motivated by the former approach, for a general absolutely continuous reference measure $q$, we introduce the following operator.

\begin{definition}[Fixed-point operator]
	Let $\mu,\nu\in\Prob_1(\RR)$ and $q\in\Prob(\RR)$ satisfy $q\ll\lambda$. We define $\Acal_q:\CDF\to\CDF$ by
	\begin{equation}
		\label{def:fixed_point_operator}
		\Acal_qF:=F_\mu\circ\bigl(q\star(Q_\nu\circ(q\ast F))\bigr).
	\end{equation}
\end{definition}

We shall repeatedly use the fact that $\Acal_q$ is shift-invariant and monotone: for every $F,G\in\CDF$ and $c,x\in\RR$,
\[
(\Acal_q(F(\cdot+c)))(x)=(\Acal_qF)(x+c); \qquad F\leq G \ \Longrightarrow\ \Acal_qF\leq\Acal_qG.
\]
These elementary properties are proved in the Gaussian case in \cite[Lemma~3.2]{AcMaPa23}, and the same arguments apply to general $q$.

The quantile formulation was extended to arbitrary absolutely continuous reference measures in \cite[Theorem~1.4]{AcMa26_semidiscrete}. More precisely, a $q$-Bass martingale from $\mu$ to $\nu$ exists if and only if there is a distribution $\alpha\in\Prob(\RR)$ such that
\begin{equation}
\label{eq:fixed-point_equiv}
\int_\RR Q_\nu\left(\int_0^1F_q\bigl(Q_\alpha(u)-Q_\alpha(v)+z\bigr)\,dv\right)\rho_q(z)\,dz=Q_\mu(u), \qquad \text{$\lambda$-a.e. }u\in(0,1),
\end{equation}
and $q\star T$ is increasing $\alpha$-a.s., where $T=Q_\nu\circ F_{\alpha\ast q}$.
Under suitable assumptions, such as conditions~(i) and~(ii) of Theorem~\ref{thm:convergence}, a direct adaptation of \cite[Theorem~2.1]{CoHe21} yields the following equivalent formulation:
\begin{equation}
\label{eq:fixed_point_eq}
F_\alpha=\Acal_qF_\alpha.
\end{equation}
The latter formulation gives rise to the fixed-point iteration studied in Section~\ref{sec:convergence}.

\subsection{Standing assumptions}

To prove uniqueness and stability, we work under the following regularity assumptions, already used in \cite{AcMa26_semidiscrete}. Assumption~\ref{ass:A3} is equivalent to the integrability condition appearing in \cite{HaJoLoObPa26} for the existence of a (Gaussian) Bass martingale attaining the variational value in \eqref{def:WT^q_S}.

\begin{assumption}[Regularity of the distributions]
\label{ass:technical-assumptions}
Let $\nu\in\Prob_1(\RR)$ satisfy $\mean(\nu)=0$, and let $q\in\Prob(\RR)$. We assume that:
\begin{enumerate}[label=(A\arabic*)]
	\item \label{ass:A1} the probability measures $\nu$ and $q$ are absolutely continuous with respect to $\lambda$;
	\item \label{ass:A2} the supports of $\nu$ and $q$ are intervals, and their densities are strictly positive $\lambda$-a.e. on the interior of their respective support;
	\item \label{ass:A3} there exist $a>1$ and $b>\frac{a}{a-1}$ such that $\nu\in\Prob_a(\RR)$ and $q\in\Prob_b(\RR)$;
	\item \label{ass:A4} there exists $\varsigma>1$ such that $Q_\nu'\in L^\varsigma(0,1)$ and $\rho_q\in L^{\max\left(\frac{2\varsigma-1}{\varsigma-1},\frac{a}{a-1}\right)}(\RR)$.
\end{enumerate}
\end{assumption}

Since $\rho_q\in L^1(\RR)$, interpolation shows that the exponent in Assumption~\ref{ass:A4} simultaneously yields the two integrability requirements $\rho_q\in L^{(2\varsigma-1)/(\varsigma-1)}(\RR)$ and $\rho_q\in L^{a/(a-1)}(\RR)$.

The following endpoint conditions ensure that every fixed-point distribution is compactly supported and are used in the uniqueness, stability, and convergence analysis.

\begin{assumption}[Conditions on the quantile endpoints]
\label{ass:technical-assumptions-2}
Let $\mu,\nu\in\Prob(\RR)$. We assume that
\begin{enumerate}[label=(A\arabic*)]
  \setcounter{enumi}{4}
	\item \label{ass:A5} $\lim_{u\to0^+}Q_\mu(u)\neq\lim_{u\to0^+}Q_\nu(u)$;
	\item \label{ass:A6} $\lim_{u\to1^-}Q_\mu(u)\neq\lim_{u\to1^-}Q_\nu(u)$.
\end{enumerate}
\end{assumption}

\section{Existence and Uniqueness of Fixed-Point Solutions}
\label{sec:existence-uniqueness}

\subsection{The $q$-Bass operator}

The aim of this subsection is to extend the geometric parametrization developed in \cite[Sections~2 and~3]{AcMa26_semidiscrete} beyond the semidiscrete setting. Suppose that $\mu$ is $n$-atomic, that is $
\mu=\sum_{i=1}^n p_i\delta_{x_i}$,
where $p_i>0$ and $x_1<\cdots<x_n$, and set $p_i^*:=\sum_{j=1}^i p_j$.
Recall that a convex polygonal chain on $[0,1]$ is the graph of a continuous convex piecewise affine function, obtained by joining each pair of consecutive vertices by a line segment. The graph of $U_\mu$ is such a chain, with intermediate vertices having abscissas $p_1^*,\dots,p_{n-1}^*$. Starting from the fixed-point equation \eqref{eq:fixed-point_equiv}, the problem can be reduced to a finite-dimensional system in the gaps between the atoms of the fixed-point distribution. The resulting map, called the $n$-atomic $q$-Bass map, parametrizes exactly the convex polygonal chains arising as integrated quantile functions of measures $\mu$ with the prescribed weights for which $(\mu,\nu)$ is irreducible; see \cite[Theorem~3.5]{AcMa26_semidiscrete}.

For a general initial marginal $\mu$, the graph of $U_\mu$ need no longer be a polygonal chain. Nevertheless, the same geometric description remains valid: by \cite[Proposition~1.8]{AcMa26_semidiscrete}, the integrated quantile functions of measures $\mu$ such that $(\mu,\nu)$ is irreducible are precisely the convex curves having the same endpoints as $U_\nu$ and lying strictly above it on $(0,1)$. This suggests replacing the finite-dimensional vector of the gaps between the atoms of the fixed-point distribution by its quantile function $Q$. The operator introduced below is obtained directly from \eqref{eq:fixed-point_equiv} and plays the analogous role of the $n$-atomic $q$-Bass map in the current infinite-dimensional setting. We will show that every such integrated quantile function $U_\mu$ can be represented as  $\Mcal(Q)$ for some quantile function $Q$, where $\Mcal$ is the operator defined below, while the semidiscrete parametrization is recovered when $Q$ is piecewise constant.

\begin{definition}[$q$-Bass operator with respect to $\nu$]
\label{def:q-Bass-operator}
Let $\nu \in \Prob_1(\RR)$ and $q \in \Prob(\RR)$ be such that $q \ll \lambda$ and $\nu$ is not a Dirac measure. The operator $\Mcal \colon \Mon((0,1), \RR) \to C([0,1])$ defined by
\begin{equation}
\label{def:q-Bass-operator}
\Mcal(Q)(p) = \int_\RR Q_\nu\left( \int_0^1 F_q(z - Q(v))\,dv \right)\int_0^p \rho_q(z - Q(u))\,du\, dz,
\end{equation}
for every $p \in [0,1]$, is called the \emph{$q$-Bass operator} with respect to $\nu$.
\end{definition}

When the terminal distribution $\nu$ is clear from the context, we simply refer to $\Mcal$ as the $q$-Bass operator.
If $Q$ is a quantile function, the fixed-point equation \eqref{eq:fixed-point_equiv} is equivalent to
\begin{equation}
\label{eq:fixed_point_new_eq}
    \Mcal(Q)=U_\mu.
\end{equation}
This formulation recasts the existence problem for $q$-Bass martingales in terms of integrated quantile functions, which provide a natural geometric interpretation of convex order and irreducibility; see Section \ref{sec:cvx-order}.
The next proposition describes the main geometric properties of $\Mcal$, showing in particular that its image consists of convex curves lying above $U_\nu$.
These properties allow us to study the existence problem geometrically, by checking whether the curve $U_\mu$ can be realized as the image under $\Mcal$ of a quantile function.

\begin{proposition}
\label{prop:Bass_operator_properties}
Let $\nu \in \Prob_1(\RR)$ and $q \in \Prob(\RR)$ be such that $q \ll \lambda$ and $\nu$ is not a Dirac measure. Let $\Mcal$ be the $q$-Bass operator with respect to $\nu$. Then, for every $Q \in \Mon((0,1), \RR)$, the function $\Mcal(Q)$ is well defined, convex, and dominates $U_\nu$. In particular, $\Mcal(Q)(0)=0$ and $\Mcal(Q)(1)=\mean(\nu)$. Moreover, the following properties hold:
\begin{enumerate}[label=(\roman*)]
\item If $\Delta Q(p):=Q(p^+)-Q(p)\geq\diam(\supp(q))$, then $\Mcal(Q)(p)=U_\nu(p)$. If $\nu$ is atomless, the converse implication also holds.
\item If $Q$ is constant on a subinterval of $[0,1]$, then $\Mcal(Q)$ is affine on that interval.
    \item $\Mcal(Q) = \Mcal(Q+c)$, for any $c \in \RR$.
\end{enumerate}
\end{proposition}

\begin{proof}
Let $Q \in \Mon((0,1), \RR)$. We first show that $\Mcal(Q)$ is well defined. For every $p \in [0,1]$,
\begin{align*}
|\Mcal(Q)(p)|
&\leq
\int_\RR
\left|
Q_\nu\left( \int_0^1 F_q(z - Q(v))\,dv \right)
\right|
\int_0^p \rho_q(z - Q(u))\,du \, dz \\
&\leq
\int_\RR
\left|
Q_\nu\left( \int_0^1 F_q(z - Q(v))\,dv \right)
\right|
\int_0^1 \rho_q(z - Q(u))\,du \, dz =
\int_0^1 |Q_\nu(u)|\,du
< \infty.
\end{align*}
By Fubini's theorem and a change of variable in $z$, we also obtain
\begin{align*} 
		\Mcal(Q)(p) &= \int_\RR \int_0^p Q_\nu\left( \int_0^1 F_q(z - Q(v))\,dv \right) \rho_q(z - Q(u))\,du \, dz \\
		&= \int_\RR \int_0^1 Q_\nu\left( \int_0^1 F_q(z + Q(u) - Q(v))\,dv \right) \rho_q(z)\mathbf{1}_{[0,p]}(u)\,du \, dz,
\end{align*}
which yields the shift-invariance property $(iii)$. Additionally, by the dominated convergence theorem, $[0,1] \ni p \mapsto \Mcal(Q)(p)$ is a continuous function.

We next prove that $\Mcal(Q)$ is convex. Let $0 \leq p_1 < p_2 < p_3 \leq 1$ be such that
\[
p_3-p_2 = p_2-p_1 = \delta.
\]
We need to show that $ \Mcal(Q)(p_2)-\Mcal(Q)(p_1) \leq \Mcal(Q)(p_3)-\Mcal(Q)(p_2)$. Now,
\begin{align*}
\Mcal(Q)(p_2)-\Mcal(Q)(p_1)
=
\int_\RR \int_{p_1}^{p_2}
Q_\nu\left( \int_0^1 F_q(z + Q(u) - Q(v))\,dv \right)
\rho_q(z)\,du \, dz.
\end{align*}
Since $Q$, $F_q$, and $Q_\nu$ are non-decreasing, it follows that
\begin{align*}
&\int_\RR \int_{p_1}^{p_2}
Q_\nu\left( \int_0^1 F_q(z + Q(u) - Q(v))\,dv \right)
\rho_q(z)\,du \, dz \\
&\qquad \leq
\int_\RR \int_{p_2}^{p_3}
Q_\nu\left( \int_0^1 F_q(z + Q(u) - Q(v))\,dv \right)
\rho_q(z)\,du \, dz  =
\Mcal(Q)(p_3)-\Mcal(Q)(p_2).
\end{align*}
Therefore $\Mcal(Q)$ is convex. 
Moreover, if $Q$ is constant on $[p_1,p_3]$, then equality holds. Consequently, $\Mcal(Q)$ is affine on $[p_1,p_3]$.

Finally, let $p\in(0,1)$. Setting
\[
A_p(z,u):=\int_0^p F_q\bigl(z+Q(u)-Q(v)\bigr)\,dv,\qquad R_p(z,u):=\int_p^1 F_q\bigl(z+Q(u)-Q(v)\bigr)\,dv,
\]
we have
\[
\Mcal(Q)(p)-U_\nu(p)=\int_0^p\int_\RR\left[Q_\nu\bigl(A_p(z,u)+R_p(z,u)\bigr)-Q_\nu\bigl(A_p(z,u)\bigr)\right]q(dz)\,du.
\]
Hence $\Mcal(Q)(p)=U_\nu(p)$ if and only if the integrand vanishes for $q(dz)\,du$-a.e. $(z,u)\in\RR\times(0,p)$. In particular, equality holds whenever $R_p(z,u)=0$ for $q(dz)\,du$-a.e. $(z,u)$, which, by Tonelli's theorem, is equivalent to
\[
F_q\bigl(z+Q(u)-Q(v)\bigr)=0
\]
for $q(dz)\,du\,dv$-a.e. $(z,u,v)\in\RR\times(0,p)\times(p,1)$. Since $q\ll\lambda$ and $Q$ is non-decreasing and left-continuous, this is equivalent to
\[
\sup\supp(q)+Q(p)-Q(p^+)\leq\inf\supp(q),
\]
or, equivalently,
\[
\Delta Q(p)\geq\diam\bigl(\supp(q)\bigr).
\]
Thus $\Delta Q(p)\geq\diam(\supp(q))$ implies $\Mcal(Q)(p)=U_\nu(p)$. If, in addition, $\nu$ is atomless, then $Q_\nu$ is increasing and the converse also holds.
\end{proof}

We conclude this section by showing that $\Mcal$ reduces to the $n$-atomic $q$-Bass map when its argument is piecewise constant.

\begin{remark}[Connection with the $n$-atomic $q$-Bass map]
\label{rmk:q-bass-operator-gen}
Let $p_1,\dots,p_n>0$ satisfy $\sum_{i=1}^n p_i=1$, and set $p_i^*:=\sum_{j=1}^i p_j$. Suppose that $Q\in\Mon((0,1),\RR)$ is constant on each interval of the partition
\[
(0,p_1^*],\qquad (p_i^*,p_{i+1}^*],\quad i=1,\dots,n-2,\qquad (p_{n-1}^*,1).
\]
Then $Q$ is determined up to an additive constant by the gap vector $h\in\RR_{\geq0}^{n-1}$ defined by
\[
h_i:=\Delta Q(p_i^*),\qquad i=1,\dots,n-1.
\]
Let $f$ be the $n$-atomic $q$-Bass map associated with the weights $p_1,\dots,p_n$ and the terminal distribution $\nu$, as introduced in \cite[Definition~2.3]{AcMa26_semidiscrete}. Then
\[
\Mcal(Q)(p_i^*)=f_i(h),\qquad i=1,\dots,n-1,
\]
and $\Mcal(Q)$ is affine between consecutive points $p_i^*$. Consequently, $\Mcal(Q)$ coincides with the polygonal chain generated by $f(h)$ in the sense of \cite[Definition~3.4]{AcMa26_semidiscrete}. By the translation invariance property of $\Mcal$, the same conclusion holds for $Q+c$, for every $c\in\RR$.
\end{remark}

\subsection{Existence of Fixed-Point Solutions}
The $q$-Bass operator introduced in Definition~\ref{def:q-Bass-operator} provides a convenient formulation of the existence problem. When $(\mu, \nu)$ is irreducible, we approximate $\mu$ by measures $(\mu_n)_{n \in \NN}$ supported on an increasing number of atoms and such that the pairs $(\mu_n ,\nu)$ are still irreducible. We then solve the corresponding semidiscrete problems, and extract a limit thanks to a generalization of Helly's selection theorem. On the other hand, when $(\mu, \nu)$ is not irreducible, we solve the problem on the irreducible components (see Remark~\ref{rmk:irreducible-components}) and suitably combine the resulting solutions.

\begin{theorem}
\label{thm:general-existence-convex-param} 
Under the assumptions of Theorem \ref{thm:q-bass-existence}, suppose in addition that $\nu$ is not a Dirac measure.
Let $\Mcal$ be the $q$-Bass operator with respect to $\nu$.
Then \eqref{eq:fixed_point_new_eq} admits a solution $Q \in \Mon((0,1), \RR)$.
Moreover, $Q$ may be chosen so that it is $\lambda$-a.e.\ constant on every level set of $Q_\mu$.
\end{theorem}

\begin{proof}
We split the proof in two steps, where we prove the statement under assumptions \ref{it:existence-1} and \ref{it:existence-2}, respectively.

\noindent\emph{Step 1: Let \ref{it:existence-1} hold, i.e. $(\mu, \nu)$ be irreducible.} If $\mu$ is $n$-atomic, the claims follow immediately from \cite[Theorem 1.9]{AcMa26_semidiscrete}, \cite[Theorem 1.4]{AcMa26_semidiscrete}, and Remark~\ref{rmk:q-bass-operator-gen}. If $\mu$ is not $n$-atomic, we can invoke Proposition~\ref{prop:approx_mu} to ensure the existence of a sequence of measures $(\mu_k)_{k \in \NN}$ such that $\mu_k$ is $k$-atomic with equal weights, $U_{\mu_k} \rightarrow U_{\mu}$ pointwise on $[0,1]$,  and $(\mu_k, \nu)$ is irreducible.
For every $k\in\NN$, let $\alpha_k$ be a Bass distribution associated with a $q$-Bass martingale from $\mu_k$ to $\nu$, whose existence follows from \cite[Theorem~1.9]{AcMa26_semidiscrete}, and let $Q_k:=Q_{\alpha_k}$. By translation invariance, we may assume that $Q_k(1/2)=0$. Then 
\begin{equation} \label{eq:approx-condition}
    \Mcal(Q_k)=U_{\mu_k}. 
\end{equation}
Moreover, $Q_k$ is $\lambda$-a.e.\ constant on every level set of $Q_{\mu_k}$. Indeed, setting $S_k:=q\star T_k$, where $T_k:=Q_\nu\circ F_{\alpha_k\ast q}$, we have 
\[ 
S_k(Q_k(u))=Q_{\mu_k}(u) 
\] 
for $\lambda$-a.e.\ $u\in(0,1)$, while $S_k$ is increasing $\alpha_k$-a.s. Hence two distinct values of $Q_k$ cannot be mapped by $S_k$ to the same value of $Q_{\mu_k}$.

Now, we invoke Proposition~\ref{prop:helly-selection-thm} to conclude the existence of $u_1^*, u_2^* \in [0,1]$ such that $u_1^* \leq u_2^*$ and of a converging subsequence $(Q_{k_n})_{n \in \NN} \subseteq \Mon((0,1), \RR)$ with limit $Q^* \in \Mon((0,1), \overline \RR)$ such that
\[
Q^*(u)= -\infty \quad \text{if } u<u_1^*,
\qquad
Q^*(u)\in \RR \quad \text{if } u_1^*<u<u_2^*,
\qquad
Q^*(u)= +\infty \quad \text{if } u>u_2^*.
\]
In particular, since $Q_k(1/2)=0$ by construction, we have $u_1^* \leq 1/2 \leq u_2^*$.
By Proposition~\ref{prop:q-Bass-operator-stability}, if $u^*_1 > 0$ then 
\[
 \lim_{n \rightarrow \infty} \Mcal(Q_{k_n})(u_1^*) = U_\nu(u_1^*) < U_\mu(u_1^*),
\]
which contradicts \eqref{eq:approx-condition}. Similarly, if $u^*_2 < 1$, then $\lim_{n \rightarrow \infty} \Mcal(Q_{k_n})(u_2^*) < U_\mu(u_2^*)$, which again contradicts \eqref{eq:approx-condition}.
Therefore, \eqref{eq:approx-condition} implies that $u_1^*=0$ and $u_2^*=1$. By Proposition~\ref{prop:q-Bass-operator-stability}, we have
\[
	U_{\mu}(p) = \lim_{n \rightarrow \infty} \Mcal(Q_{k_n})(p) = \Mcal(Q^*)(p), \qquad \text{for all } p \in [0,1].
\]

It remains to verify that $Q^*$ is $\lambda$-a.e.\ constant on every level set of $Q_{\mu}$. Let $x$ be an atom of $\mu$ and set $I_x:=\{u\in(0,1):Q_\mu(u)=x\}$. Fix a compact interval $J$ contained in the interior of $I_x$. Since $U_{\mu_k}$ is the piecewise affine interpolation of $U_\mu$ on the grid $\{i/k:i=0,\ldots,k\}$, for all sufficiently large $k$ the function $Q_{\mu_k}$ is constant on $J$. Consequently, $Q_k$ is constant on $J$. Passing to the pointwise limit, we conclude that $Q^*$ is constant on $J$. Since $J$ is arbitrary, $Q^*$ is constant on the interior of $I_x$, and hence $\lambda$-a.e.\ on $I_x$. As the atoms of $\mu$ are at most countable, $Q^*$ is $\lambda$-a.e.\ constant on every level set of $Q_\mu$.

\noindent\emph{Step 2: Let \ref{it:existence-2} hold. Having proved Step 1, w.l.o.g. we can assume that $(\mu, \nu)$ is not irreducible.} Denote by $u_1 < \dots <u_{|K|}$ the elements of $K$ and let $u_0:=0$, $u_{|K|+1}:=1$, and $a_i:=u_i-u_{i-1}$. For each $i=1,\dots,|K|+1$, let $(\mu_i,\nu_i)$ be the normalized irreducible component corresponding to $(u_{i-1},u_i)$, and let $\Mcal_i$ denote the associated $q$-Bass operator. Thus
\begin{equation}
\label{eq:irreducible_decomposition}
    Q_{\mu_i}(s)=Q_\mu(u_{i-1}+a_i s),\qquad Q_{\nu_i}(s)=Q_\nu(u_{i-1}+a_i s),\qquad s\in(0,1),
\end{equation}
and, by applying Step 1 to each irreducible component, we obtain $Q_i^* \in \Mon((0,1),\RR)$ such that
\[
\Mcal_i(Q_i^*)=U_{\mu_i}.
\]
Moreover, $Q_i^*$ may be chosen to be constant on each level set of $Q_{\mu_i}$.

Set $r_i:=Q_i^*(1^-)$ for $i=1,\dots,|K|$ and $\ell_i:=Q_i^*(0^+)$ for $i=2,\dots,|K|+1$. By \eqref{eq:compact-support-components} and Proposition~\ref{prop:compact-support-fixed-point} below,  these quantities are finite and, therefore, the functions $Q_i^*$ are bounded. Choose $c_1:=0$ and, recursively,
\[
c_{i+1}:=c_i+r_i-\ell_{i+1}+D,\qquad i=1,\dots,|K|,
\]
where $D = \diam(\supp(q))$. Extending $Q_i^*$ at $1$ by $Q_i^*(1):=Q_i^*(1^-)$ for $i=1,\dots,|K|$, define $\hat Q:(0,1)\to\RR$ by
\begin{equation}
\label{eq:Q_hat_reducible}
\hat Q(u):=\begin{cases}Q_i^*\left(\dfrac{u-u_{i-1}}{a_i}\right)+c_i,&u\in(u_{i-1},u_i],\quad i=1,\dots,|K|, \\ 
Q_{|K|+1}^*\left(\dfrac{u-u_{|K|}}{a_{|K|+1}}\right)+c_{|K|+1},&u\in(u_{|K|},1).
\end{cases}
\end{equation}
Then $\hat Q$ is non-decreasing and left-continuous, and
\[
\Delta\hat Q(u_i)=\hat Q(u_i^+)-\hat Q(u_i)=D,\qquad i=1,\dots,|K|.
\]

We now verify that $\hat Q$ solves the global problem. By construction, $\Delta\hat Q(u_i)=D$ for every $i=1,\dots,|K|$. Hence, Proposition~\ref{prop:Bass_operator_properties} and the definition of $K$ yield
\[
\Mcal(\hat Q)(u_i)=U_\nu(u_i)=U_\mu(u_i).
\]
Let $p=u_{i-1}+a_i t$ for some $i\in\{1,\dots,|K|+1\}$ and $t\in[0,1]$. The gaps of length $D$ imply that, whenever $u$ belongs to the $i$-th irreducible component, the terms corresponding to earlier components contribute $1$ to the inner distribution function, while those corresponding to later components contribute $0$. Thus, writing $u=u_{i-1}+a_i s$,
\[
\int_0^1F_q\bigl(z+\hat Q(u)-\hat Q(v)\bigr)\,dv=u_{i-1}+a_i\int_0^1F_q\bigl(z+Q_i^*(s)-Q_i^*(r)\bigr)\,dr.
\]
Therefore, after rescaling,
\[
\Mcal(\hat Q)(p)-\Mcal(\hat Q)(u_{i-1})=a_i\Mcal_i(Q_i^*)(t)=a_iU_{\mu_i}(t).
\]
Since $\Mcal(\hat Q)(u_{i-1})=U_\mu(u_{i-1})$ and $U_\mu(p)-U_\mu(u_{i-1})=a_iU_{\mu_i}(t)$, we conclude that
\[
\Mcal(\hat Q)(p)=U_\mu(p),\qquad p\in[0,1].
\]
In particular, since each $Q_i^*$ is constant on every level set of $Q_{\mu_i}$, the reparametrization in \eqref{eq:irreducible_decomposition} and \eqref{eq:Q_hat_reducible} imply that $\hat Q$ is constant on the intersection of every level set of $Q_\mu$.
\end{proof}

We can now prove the existence of the
$q$-Bass martingales stated in the Introduction.

\begin{proof}[Proof of Theorem~\ref{thm:q-bass-existence}]
If $\nu$ is a Dirac measure, then $\mu \preceq_c \nu$ implies $\mu=\nu$. Thus, the existence of a $q$-Bass martingale is immediate: when both marginals are Dirac measures, the product coupling $\mu\otimes\nu$ is itself a $q$-Bass martingale. So w.l.g. we restrict to the case of $\nu$ not Dirac.

By \cite[Theorem~1.4]{AcMa26_semidiscrete}, the existence of a $q$-Bass martingale from $\mu$ to $\nu$ is equivalent to the existence of a function $Q\in\Mon((0,1),\RR)$ such that
\begin{equation}
\label{mainFixedPointEq}
\int_\RR Q_\nu\left(\int_0^1F_q\bigl(Q(u)-Q(v)+z\bigr)\,dv\right)\rho_q(z)\,dz=Q_\mu(u),\qquad\text{for $\lambda$-a.e. }u\in(0,1),
\end{equation}
and the map
\[
S_Q(x):=\int_\RR Q_\nu\left(\int_0^1F_q\bigl(x-Q(v)+z\bigr)\,dv\right)\rho_q(z)\,dz
\]
is increasing $\alpha$-a.s., where $\alpha:=Q_\#\lambda_{|(0,1)}$.
By the change of variables $z\mapsto z-Q(u)$, equation \eqref{mainFixedPointEq} is equivalent to
\[
\int_\RR Q_\nu\left(\int_0^1F_q\bigl(z-Q(v)\bigr)\,dv\right)\rho_q\bigl(z-Q(u)\bigr)\,dz=Q_\mu(u),\qquad\text{for $\lambda$-a.e. }u\in(0,1).
\]
Integrating over $(0,p)$, we conclude that \eqref{mainFixedPointEq} is equivalent to
\[
\int_0^p\int_\RR Q_\nu\left(\int_0^1F_q\bigl(z-Q(v)\bigr)\,dv\right)\rho_q\bigl(z-Q(u)\bigr)\,dz\,du=U_\mu(p),\qquad p\in[0,1].
\]
By Fubini's theorem,
\[
\begin{aligned}
\int_0^p\int_\RR Q_\nu&\left(\int_0^1F_q\bigl(z-Q(v)\bigr)\,dv\right)\rho_q\bigl(z-Q(u)\bigr)\,dz\,du\\
&=\int_\RR Q_\nu\left(\int_0^1F_q\bigl(z-Q(v)\bigr)\,dv\right)\int_0^p\rho_q\bigl(z-Q(u)\bigr)\,du\,dz = \Mcal(Q)(p).
\end{aligned}
\]
Conversely, if $\Mcal(Q)=U_\mu$, then both sides are absolutely continuous on $[0,1]$, and differentiating this identity yields \eqref{mainFixedPointEq} for $\lambda$-a.e. $u\in(0,1)$. Therefore, \eqref{mainFixedPointEq} is equivalent to  \eqref{eq:fixed_point_new_eq}.

By Theorem~\ref{thm:general-existence-convex-param}, there exists $Q\in\Mon((0,1),\RR)$ satisfying \eqref{eq:fixed_point_new_eq} and constant $\lambda$-a.e. on the level sets of $Q_\mu$. Since $S_Q(Q(u))=Q_\mu(u)$
for almost every $u\in(0,1)$, we may choose a set $E\subseteq(0,1)$ of full measure under $\lambda$ on which this identity holds and $Q$ is constant on each level set of $Q_\mu$. If $u,v\in E$ and $Q(u)<Q(v)$, then $u<v$ and $Q_\mu(u)\leq Q_\mu(v)$, with equality excluded by the constancy property. Hence, $S_Q$ is increasing on the $\alpha$-full set $Q(E)$, where $\alpha=Q_\#\text{Unif}_{(0,1)}$. Thus, \eqref{mainFixedPointEq} holds, and \cite[Theorem~1.4]{AcMa26_semidiscrete} yields a $q$-Bass martingale from $\mu$ to $\nu$.
\end{proof}

\subsection{Uniqueness of Fixed-Point Solutions}

In the semidiscrete setting, uniqueness relies on the fact that the $n$-atomic $q$-Bass map is the gradient of a strictly concave potential. Under Assumption~\ref{ass:technical-assumptions}, this yields the injectivity of the map and hence uniqueness, up to translation, of the corresponding Bass distribution; see \cite[Proposition~2.9 and Theorem~1.12]{AcMa26_semidiscrete}.
To extend this argument to general initial marginals, in the proof of Theorem~\ref{thm:uniqueness-fixed-point-sol} we introduce an infinite-dimensional potential associated with the $q$-Bass operator with respect to $\nu$. We show that this potential is twice Fr\'echet differentiable, that its first differential recovers $\Mcal$, and that its second differential is negative definite modulo constant directions. Under Assumptions~\ref{ass:technical-assumptions} and~\ref{ass:technical-assumptions-2}, this yields the injectivity of $\Mcal$ modulo additive constants on the relevant class of quantile functions, and hence the uniqueness, up to translation, of the fixed-point distribution.

\begin{proof}[Proof of Theorem~\ref{thm:uniqueness-fixed-point-sol}]
By Assumptions~\ref{ass:technical-assumptions},\ref{ass:technical-assumptions-2} and Proposition~\ref{prop:compact-support-fixed-point} below, every fixed-point distribution is compactly supported. Hence, when studying the solutions, we may restrict the operator $\Mcal$ to
\[
    \Mon((0,1), \RR) \cap L^\infty(0,1).
\]
We now introduce the functional
\[
    \Vcal: L^\infty(0,1) \to \RR,
    \qquad
    \Vcal (Q) = \int_\RR U_\nu\left( \int_0^1 F_q(z-Q(v))\, dv \right) dz.
\]
This is the infinite-dimensional analogue of the potential considered in \cite[Proposition~2.9]{AcMa26_semidiscrete}. Exactly as in the proof of \cite[Proposition~2.9]{AcMa26_semidiscrete}, Assumption~\ref{ass:A3} guarantees that $\Vcal(Q)$ is well defined for every $Q \in L^\infty(0,1)$.

Let $Q, f \in L^\infty(0,1)$. By the same differentiation argument used in \cite[Proposition~2.9]{AcMa26_semidiscrete}, together with \cite[Lemma A.1]{AcMa26_semidiscrete} and Fubini's theorem,  we find that the Gâteaux derivative of $\Vcal$ at $Q$ along the direction $f$ is given by
\begin{equation}
\label{eq:gateaux_deriv}
    \begin{aligned}
    \frac{d}{d\varepsilon}\Vcal(Q+\varepsilon f)\bigg|_{\varepsilon=0}
    &=
    -\int_\RR Q_\nu\left(\int_0^1 F_q(z-Q(v))\, dv\right)\int_0^1 \rho_q(z-Q(u))f(u)\,du\,dz \\
    &=
    -\int_{[0,1]\times\RR}
    Q_\nu\left(\int_0^1 F_q(z+Q(u)-Q(v))\,dv\right)\rho_q(z)f(u)\,dz\,du,
\end{aligned}
\end{equation}
where in the second line we performed the change of variables $z \mapsto z+Q(u)$.
This expression defines a bounded linear functional of $f$. Moreover, adapting the proof of \cite[Lemma~2.8]{AcMa26_semidiscrete}, the assumptions $\nu\in\Pcal_a(\RR)$ and $\rho_q\in L^{\frac{a}{a-1}}(\RR)$ imply that
\[
    \sup_{\|f\|_\infty\leq1}\left|\frac{d}{d\varepsilon}\Vcal(Q_n+\varepsilon f)\Big|_{\varepsilon=0}-\frac{d}{d\varepsilon}\Vcal(Q+\varepsilon f)\Big|_{\varepsilon=0}\right|\longrightarrow0
\]
whenever $Q_n\to Q$ in $L^\infty(0,1)$. Thus, the Gâteaux derivative in \eqref{eq:gateaux_deriv} depends continuously on $Q$ in the norm of $(L^\infty(0,1))^*$. Therefore, $\Vcal$ is Fr\'echet differentiable, with differential
\[
    D_Q\Vcal(f)
    =
    -\int_{[0,1]\times\RR} Q_\nu\Bigg(\int_0^1 F_q(z+Q(u)-Q(v))  dv\Bigg)\rho_q(z)f(u)\,dz\,du.
\]

We now compute the second-order differential. Differentiating  $D_Q\Vcal(f)$ in the direction $g \in L^\infty(0,1)$, and arguing exactly as in \cite[Lemma~2.8]{AcMa26_semidiscrete},  we obtain the Gâteaux derivative of $D_Q\Vcal(f)$ at $Q$ along the direction $g$:
\begin{equation}
\label{eq:gateaux_deriv_2}
    \begin{aligned}
    \frac{d}{d \varepsilon} D_{Q+\varepsilon g}\Vcal(f)\bigg |_{\varepsilon =0}
    = - \int_{[0,1]^2 \times \RR} & Q'_\nu\Bigg( \int_0^1 F_q(z+Q(u)-Q(v))  dv \Bigg)  \\
    &\times \rho_q(z) \rho_q(z+Q(u)-Q(w)) f(u)[g(u)-g(w)]\, dz\, du\, dw.
\end{aligned}
\end{equation}
Denote by $B_Q(f,g)$ the right-hand side of \eqref{eq:gateaux_deriv_2}. Assumption~\ref{ass:A4} and the estimates used in the proof of \cite[Proposition~2.9]{AcMa26_semidiscrete} show that $B_Q$ is a bounded bilinear form and that
\[
\sup_{\substack{\|f\|_\infty \leq 1 \\ \|g\|_\infty \leq 1}} |B_{Q_n}(f,g) - B_Q(f,g)| \longrightarrow 0
\]
whenever $Q_n \to Q$ in $L^\infty(0,1)$. Hence $Q\mapsto D_Q\Vcal$ is Frèchet differentiable, with $D_Q^2\Vcal=B_Q$.
After the change of variables $z\mapsto z+Q(u)$, we may therefore write
\[
D^2_Q\Vcal(f,g)=-\int_{[0,1]^2}f(u)[g(u)-g(w)]K_Q(u,w)\,du\,dw,
\]
where we set
\[
K_Q(u,w):=\int_\RR Q'_\nu\Bigg(\int_0^1F_q(z-Q(v))\,dv\Bigg)\rho_q(z-Q(u))\rho_q(z-Q(w))\,dz.
\]
We now take $g=f$. Since $K_Q$ is symmetric, we can symmetrize the previous expression and obtain
\[
D^2_Q\Vcal(f,f)=-\frac12\int_{[0,1]^2}|f(u)-f(w)|^2K_Q(u,w)\,du\,dw.
\]
By Assumption~\ref{ass:technical-assumptions}, $K_Q\geq0$ and
\begin{equation}
\label{eq:connection_equation_differetial}
K_Q(u,w)>0\qquad\text{for a.e. $(u,w)\in[0,1]^2$ such that }|Q(u)-Q(w)|<\diam(\supp(q)).
\end{equation}
In particular, if
\[
Q\in\Mon((0,1),\RR)\cap L^\infty(0,1)\qquad\text{and}\qquad\sup_{p\in(0,1)}\Delta Q(p)<\diam(\supp(q)),
\]
\eqref{eq:connection_equation_differetial} implies that, for any fixed $u\in(0,1)$, the map $w\mapsto K_Q(u,w)$ is positive almost everywhere in a neighbourhood of $u$. 
Hence
\[
D^2_Q\Vcal(f,f)\leq0,
\]
with equality if and only if $f$ is constant almost everywhere. Therefore, $Q\mapsto D_Q\Vcal$ is strictly monotone and hence injective on
\[
\Dcal:=\Mon((0,1),\RR)\cap\left\{Q\in L^\infty(0,1):\sup_{p\in(0,1)}\Delta Q(p)<\diam(\supp(q))\right\}
\]
modulo the equivalence relation $Q_1\sim Q_2$ if $Q_1-Q_2$ is constant.

	Since $(\mu, \nu)$ is irreducible, Proposition~\ref{prop:Bass_operator_properties} implies that the quantile function of any fixed-point distribution belongs to $\Dcal$. Finally, using the identification
	\[
		\Mcal(Q)(p) = -D_Q \Vcal(\mathbf 1_{[0,p]}),
	\]
	 we conclude that the fixed-point distribution is unique up to translation, since the quantile function of any fixed-point distribution must solve \eqref{eq:fixed_point_new_eq}.
\end{proof}

\section{Stability of the $q$-Bass operator and the fixed-point distribution}
\label{sec:stability}

We now investigate the stability of the $q$-Bass construction under perturbations  of the initial and terminal marginals, as well as of the reference measure. The first step is to establish the stability of the operator $\Mcal$ when the terminal quantile functions converge pointwise, the reference densities converge in $L^1$, and its arguments form a pointwise convergent sequence of monotone functions.
Since such a sequence need not be uniformly bounded, its pointwise limit may take values in the extended real line. The proposition below therefore treats both finite-valued limits and the two possible forms of divergence at the endpoints. In the latter cases, the corresponding values of $\Mcal_k(Q_k)$ converge to the contact value $U_\nu$. This will allow us to rule out non-finite limits under irreducibility and, together with uniqueness up to translation, to deduce the stability of the normalized fixed-point distributions.

\begin{proposition}[Stability of the $q$-Bass operator]
\label{prop:q-Bass-operator-stability}
Let $\nu \in \Prob_1(\RR)$ and $q\in \Prob(\RR)$ be such that $q \ll \lambda$ and $\nu$ is not a Dirac measure. Let $(\nu_k)_{k\in\NN}\subseteq \Prob_1(\RR)$ and $(q_k)_{k\in\NN}\subseteq \Prob(\RR)$ be such that
\[
Q_{\nu_k}\to Q_\nu \quad\text{pointwise},
\qquad
|Q_{\nu_k}|\le H \text{ for some  $H\in L^1(0,1)$\  for all }k \in \NN,
\]
and
\[
\rho_{q_k}\to \rho_q \quad\text{in } L^1(\RR), \qquad q_k \ll \lambda\ \text{for all }k \in \NN.
\]
For every $k\in\NN$, let $\Mcal_k$ denote the $q_k$-Bass operator with respect to $\nu_k$, and let $\Mcal$ denote the $q$-Bass operator with respect to $\nu$.
Moreover, let $(Q_k)_{k\in\NN}\subseteq \Mon((0,1),\RR)$ be a sequence converging pointwise to some $Q\in \Mon((0,1),\overline{\RR})$. Then the following assertions hold:
\begin{enumerate}[label = (\roman*)]
    \item \label{it:q-Bass-operator-stability-1} If $Q\in \Mon((0,1),\RR)$, then
    \[
    	\Mcal_k(Q_k)(p)\longrightarrow \Mcal(Q)(p), \qquad \text{for every } p\in[0,1].
    \]

    \item \label{it:q-Bass-operator-stability-2}  If there exists $u_1^*\in(0,1)$ such that
    \[
    	Q(u)=-\infty \quad \text{for every } u<u_1^*, \qquad Q(u)>-\infty \quad \text{for every } u>u_1^*,
    \]
    then
    \[
    	\Mcal_k(Q_k)(u_1^*)\longrightarrow U_\nu(u_1^*).
    \]

    \item \label{it:q-Bass-operator-stability-3}  If there exists $u_2^*\in(0,1)$ such that
    \[
    	Q(u)<+\infty \quad \text{for every } u<u_2^*, \qquad Q(u)=+\infty \quad \text{for every } u>u_2^*,
    \]
    then
    \[
    	\Mcal_k(Q_k)(u_2^*)\longrightarrow U_\nu(u_2^*).
    \]
\end{enumerate}
\end{proposition}

\begin{proof}
By the dominated convergence theorem,
\[
	\|Q_{\nu_k}-Q_\nu\|_{L^1(0,1)}\longrightarrow 0.
\]
Moreover,
\[
\sup_{x\in\RR}|F_{q_k}(x)-F_q(x)| \leq \|\rho_{q_k}-\rho_q\|_{L^1(\RR)}\longrightarrow 0.
\]
In particular, $q_k$ converges weakly to $q$. By the Skorokhod representation theorem, up to enlarging the underlying probability space, we may assume that there exist random variables $(Z_k)_{k\in\NN}$ and $Z$ such that
\[
Z_k\sim q_k,\qquad Z\sim q,\qquad Z_k\to Z \quad \text{a.s.}
\]

For every $k\in\NN$, every $p\in(0,1]$, and every $x\in\RR$, define
\[
I_{k,p}(x):=\int_0^p F_{q_k}(x-Q_k(v))\,dv.
\]
Whenever $Q((0,1))\subseteq\RR$, define also $I_p(x):=\int_0^p F_q(x-Q(v))\,dv$. Fix $p\in(0,1]$, let $U_p\sim \mathrm{Unif}_{[0,p]}$, and assume that $U_p$ is independent of $Z$ and of the whole family $(Z_k)_{k\in\NN}$. Set
\[
X_{k,p}:=Q_k(U_p)+Z_k,
\qquad
T_{k,p}:=I_{k,1}(X_{k,p}).
\]
By Fubini's theorem,
\begin{equation}
\label{eq:Mcal-prob-repr-stab}
\Mcal_k(Q_k)(p)=p\,\EE\bigl[Q_{\nu_k}(T_{k,p})\bigr].
\end{equation}
Since $q_k\ll\lambda$, the law of $X_{k,p}$ is absolutely continuous, with CDF
\[
F_{X_{k,p}}(x)=\frac{1}{p}I_{k,p}(x),
\qquad x\in\RR.
\]
Hence $X_{k,p}$ has density $\frac1p I_{k,p}'$, where
\[
I_{k,p}'(x)=\int_0^p \rho_{q_k}(x-Q_k(v))\,dv,
\]
for a.e.\ $x\in\RR$. In particular,
\[
0\le I_{k,p}'(x)\le I_{k,1}'(x)
\qquad\text{for a.e. }x\in\RR.
\]
Therefore, for every interval $J\subseteq(0,1)$,
\[
P(T_{k,p}\in J) = \frac 1 p\int_{\RR}  \mathbf{1}_{J}(I_{k,1}(x)) I_{k,p}'(x)\,dx \leq \frac 1 p\int_{\RR} \mathbf{1}_{J}(I_{k,1}(x)) I_{k,1}'(x)\,dx = \frac{|J|}{p}.
\]
Thus the law of $T_{k,p}$ is absolutely continuous with density bounded by $1/p$.

We shall use the following elementary fact: if $Y_k$ converges to $Y$ in law and the laws of $Y_k$ and $Y$ are absolutely continuous with densities bounded by some constant $C>0$, then
\[
\EE[g(Y_k)]\longrightarrow \EE[g(Y)]
\qquad\text{for every } g\in L^1((0,1)).
\]
Indeed, this follows by approximating $g$ in $L^1((0,1))$ with bounded continuous functions.

\medskip
\noindent\emph{Proof of~\ref{it:q-Bass-operator-stability-1}.}
Fix $p\in(0,1]$. Since $Q((0,1))\subseteq\RR$, set $X_p:=Q(U_p)+Z$ and $T_p:=I_1(X_p)$.
Again by Fubini's theorem,
\[
\Mcal(Q)(p)=p\,\EE\bigl[Q_\nu(T_p)\bigr].
\]
Moreover, the same argument as above shows that the law of $T_p$ is absolutely continuous with density bounded by $1/p$. We claim that $T_{k,p}\longrightarrow T_p$ a.s.
Indeed, for every $v\in(0,1)$,
\[
	F_{q_k}\bigl(Q_k(U_p)+Z_k-Q_k(v)\bigr) \longrightarrow F_q\bigl(Q(U_p)+Z-Q(v)\bigr) \qquad\text{a.s.}
\]
Since $0\le F_{q_k}\le 1$, the dominated convergence theorem yields
\[
T_{k,p}
=
\int_0^1 F_{q_k}\bigl(Q_k(U_p)+Z_k-Q_k(v)\bigr)\,dv
\longrightarrow
\int_0^1 F_q\bigl(Q(U_p)+Z-Q(v)\bigr)\,dv
=
T_p
\qquad\text{a.s.}
\]
Thus $T_{k,p}$ converges to $T_p$ in $L^1$ and, therefore, in measure. Therefore,
\[
\begin{aligned}
\bigl|\EE[Q_{\nu_k}(T_{k,p})]-\EE[Q_\nu(T_p)]\bigr|
&\le
\bigl|\EE[(Q_{\nu_k}-Q_\nu)(T_{k,p})]\bigr|
+\bigl|\EE[Q_\nu(T_{k,p})]-\EE[Q_\nu(T_p)]\bigr| \\
&\le
\frac1p\|Q_{\nu_k}-Q_\nu\|_{L^1(0,1)}
+\bigl|\EE[Q_\nu(T_{k,p})]-\EE[Q_\nu(T_p)]\bigr|.
\end{aligned}
\]
For $k\to\infty$, the first term tends to $0$, and the second one tends to $0$ by the observation above with $g=Q_\nu$. Using \eqref{eq:Mcal-prob-repr-stab}, we conclude that $
\Mcal_k(Q_k)(p)\longrightarrow \Mcal(Q)(p)$, for every $p\in(0,1]$.

\medskip
\noindent\emph{Proof of~\ref{it:q-Bass-operator-stability-2}.}
Set $p:=u_1^*$ and define
\[
S_{k,p}:=I_{k,p}(X_{k,p})
=
\int_0^p F_{q_k}\bigl(Q_k(U_p)+Z_k-Q_k(v)\bigr)\,dv.
\]
Since $F_{X_{k,p}}=I_{k,p}/p$ is continuous, it holds that
\[
S_{k,p}=p\,F_{X_{k,p}}(X_{k,p})\sim \mathrm{Unif}_{[0,p]}
\qquad\text{for every }k\in\NN.
\]
We claim that
\[
T_{k,p}-S_{k,p}\longrightarrow 0
\qquad\text{a.s.}
\]
Indeed, since $U_p<p$ almost surely and $Q(u)=-\infty$ for every $u<p$, one has
\[
Q_k(U_p)\longrightarrow -\infty
\qquad\text{a.s.}
\]
On the other hand, for every $v>p$ one has $Q(v)>-\infty$, hence
\[
Q_k(U_p)+Z_k-Q_k(v)\longrightarrow -\infty
\qquad\text{a.s.}
\]
and therefore
\[
F_{q_k}\bigl(Q_k(U_p)+Z_k-Q_k(v)\bigr)\longrightarrow 0
\qquad\text{a.s.}
\]
for every $v>p$. Since $0\le F_{q_k}\le 1$, the dominated convergence theorem gives
\[
T_{k,p}-S_{k,p}
=
\int_p^1 F_{q_k}\bigl(Q_k(U_p)+Z_k-Q_k(v)\bigr)\,dv
\longrightarrow 0
\qquad\text{a.s.}
\]
As $S_{k,p}\sim \mathrm{Unif}_{[0,p]}$ for every $k$, it follows that $T_{k,p}$ converges to $\mathrm{Unif}_{[0,p]}$ weakly. Now, let $U\sim \mathrm{Unif}_{[0,p]}$. Then
\[
\begin{aligned}
\left|\EE[Q_{\nu_k}(T_{k,p})]-\frac1 p U_\nu(p)\right|
&\le
\bigl|\EE[(Q_{\nu_k}-Q_\nu)(T_{k,p})]\bigr|
+\left|\EE[Q_\nu(T_{k,p})]-\EE[Q_\nu(U)]\right| \\
&\le
\frac1p\|Q_{\nu_k}-Q_\nu\|_{L^1(0,1)}
+\left|\EE[Q_\nu(T_{k,p})]-\EE[Q_\nu(U)]\right|.
\end{aligned}
\]
The first term tends to $0$, and the second one tends to $0$ by the observation above. Therefore,
\[
\Mcal_k(Q_k)(u_1^*) = p\EE[Q_{\nu_k}(T_{k,p})]
\longrightarrow U_\nu(u_1^*).
\]

\medskip

\noindent\emph{Proof of~\ref{it:q-Bass-operator-stability-3}.}
The proof is identical to that of \ref{it:q-Bass-operator-stability-2}, replacing the assumption $Q(u)=-\infty$ for $u<u_1^*$, with $Q(u)=+\infty$ for $u>u_2^*$.
One again writes
\[
T_{k,p}-S_{k,p}
=
\int_p^1 F_{q_k}\bigl(Q_k(U_p)+Z_k-Q_k(v)\bigr)\,dv,
\qquad p:=u_2^*,
\]
and observes that the integrand converges almost surely to $0$ for every $v>p$, because now $Q_k(v)\to+\infty$. The rest of the argument is unchanged, and yields
\[
\Mcal_k(Q_k)(u_2^*)\longrightarrow U_\nu(u_2^*).
\]
\end{proof}

We can now use the above proposition, together with uniqueness, to establish the stability of the fixed-point distribution stated in the Introduction.

\begin{proof}[Proof of Theorem~\ref{thm:fixed-point-stability}]
	By Theorem~\ref{thm:q-bass-existence}, there exists a $q$-Bass martingale from $\mu$ to $\nu$. By Theorem~\ref{thm:uniqueness-fixed-point-sol}, we denote by $\alpha \in \Prob(\RR)$ the unique fixed-point distribution of this $q$-Bass martingale such that $Q_\alpha(1/2)=0$.
	For any $k \in \NN$, let $\Mcal_k$ be the $q_k$-Bass operator with respect to $\nu_k$. 
    Since $\alpha_k\in\Prob(\RR)$ is a fixed-point distribution  associated with a $q_k$-Bass martingale from $\mu_k$ to $\nu_k$, we obtain (see the proof of Theorem~\ref{thm:q-bass-existence})
	\begin{equation}
    \label{eq:stability_eq_sequence}
		\Mcal_k(Q_{\alpha_k}) = U_{\mu_k}.
	\end{equation}
	Let $(Q_{\alpha_{k_n}})_{n \in \NN}$ be a subsequence of $(Q_{\alpha_k})_{k \in \NN}$. As we argued in the proof of Theorem~\ref{thm:general-existence-convex-param}
 the sequence $(Q_{\alpha_{k_n}})_{n \in \NN}$ admits a further subsequence with pointwise limit $Q^* \in \Mon((0,1), \RR)$. Otherwise,  $(Q_{\alpha_{k_n}})_{n \in \NN}$ would admit a further subsequence with pointwise limit $Q^* \in \Mon((0,1), \overline \RR)\setminus \Mon((0,1), \RR)$, in contradiction with the assumption that $(\mu, \nu)$ is irreducible.
 
In particular, since $Q^*$ is the pointwise limit of a subsequence of $Q_{\alpha_k}$, Proposition~\ref{prop:q-Bass-operator-stability} together with \eqref{eq:stability_eq_sequence} yields
\[
    \Mcal(Q^*) = \lim_{k \to \infty} \Mcal_k(Q_{\alpha_k}) = U_{\mu}.
\]
Thus, $Q^*$ is a quantile function satisfying \eqref{eq:fixed_point_new_eq}. Since the assumptions of Theorem~\ref{thm:uniqueness-fixed-point-sol} are satisfied, the corresponding Bass distribution is unique up to translation. Hence, $Q^*$ is the quantile function of the Bass distribution for the $q$-Bass martingale from $\mu$ to $\nu$.

Moreover, since $Q_{\alpha_k}(1/2)=0$ for every $k \in \NN$, we have $Q^*(1/2)=0$. Therefore, since $Q_{\alpha}(1/2)=0$, the uniqueness of the Bass distribution up to translation implies that
	\[
		Q^* = Q_{\alpha}.
	\]
	Since any subsequence of $(Q_{\alpha_k})_{k \in \NN}$ admits a further subsequence converging to $Q_\alpha$ pointwise, then $(Q_{\alpha_k})_{k \in \NN}$ converges pointwise to $Q_\alpha$. This yields the weak convergence of $\alpha_k$ to $\alpha$.
\end{proof}

\section{Support Bounds and Invariance Properties of the Fixed-Point Distribution
}\label{sec:support}

Having established existence, uniqueness, and stability of fixed-point distributions, we now investigate what the initial and terminal marginals, together with the reference measure $q$, reveal about the diameter of their supports. This information is also relevant from a numerical perspective. In particular, knowing in advance whether a fixed-point distribution is compactly supported helps determine whether the fixed-point equation can be approximated on a bounded computational domain. 

We first relate boundedness of the quantile function of fixed-point distributions at each endpoint to the corresponding endpoints of the initial and terminal quantiles. This yields sufficient conditions for compact support and, under a mild strict monotonicity assumption on $Q_\nu$, a converse characterization. In particular, compactness of the prescribed marginals alone does not guarantee compactness of the fixed-point distribution.

We then seek quantitative information on the size of the support. To this end, we introduce the $r$-maximal curves associated with the $q$-Bass operator and use them to derive lower bounds on the diameter of any fixed-point distribution. Such estimates may help determine the range and resolution of the grid used to implement the fixed-point iteration numerically. The convergence of this iteration is studied in Section~\ref{sec:convergence}. Finally, we describe how fixed-point distributions, maximal curves, and the induced $q$-Bass coupling behave under positive affine transformations of the reference measure.

\begin{proposition}[Compact support of fixed-point solutions]
\label{prop:compact-support-fixed-point}
Let $\mu,\nu \in \Prob_1(\RR)$ and $q \in \Prob(\RR)$ be such that $\mu \preceq_c \nu$, $q \ll \lambda$ and $\nu$ is not a Dirac measure. For $\eta \in \Prob(\RR)$, set $\ell_\eta := \lim_{u \to 0^+} Q_\eta(u)$ and $r_\eta := \lim_{u \to 1^-} Q_\eta(u)$.

\begin{enumerate}[label=(\roman*)]
	\item \label{it:compact-support-fixed-point-1} If $\ell_\mu \neq \ell_\nu$, then the quantile function of every fixed-point distribution is bounded from below. If, moreover, $\ell_\nu \in \RR$ and $Q_\nu$ is increasing on a neighborhood of $0$, then the converse holds too.
	
	\item \label{it:compact-support-fixed-point-2} If $r_\mu \neq r_\nu$, then the quantile function of every fixed-point distribution is bounded from above. If, moreover, $r_\nu \in \RR$ and $Q_\nu$ is increasing on a neighborhood of $1$, then the converse holds too.
\end{enumerate}
Hence, the inequalities $\ell_\mu \neq \ell_\nu$ and $r_\mu \neq r_\nu$ together imply that every fixed-point distribution is compactly supported. If, in addition, $\ell_\nu, r_\nu \in \RR$ and $Q_\nu$ is increasing in neighborhoods of both $0$ and $1$, then the reverse implication also holds.
\end{proposition}
\begin{proof}
Let $\alpha \in \Prob(\RR)$ be a fixed-point distribution, and let $Q_\alpha \in \Mon((0,1),\RR)$ be its quantile function. By \eqref{eq:fixed-point_equiv}, for every $p \in (0,1)$,
\[
\int_\RR Q_\nu\!\left(\int_0^1 F_q(z+Q_\alpha(p)-Q_\alpha(v))\,dv\right)\rho_q(z)\,dz = Q_\mu(p).
\]

We first prove the statement at the left endpoint. Assume that $Q_\alpha$ is not bounded from below, namely $Q_\alpha(p) \to -\infty$ as $p \to 0^+$. Let $p_n \downarrow 0$ and set $I_n(z) := \int_0^1 F_q(z+Q_\alpha(p_n)-Q_\alpha(v))\,dv$. Then $I_n(z) \downarrow 0$ for every $z \in \RR$, hence $Q_\nu(I_n(z)) \downarrow \ell_\nu$ for every $z \in \RR$.

We claim that $\int_\RR Q_\nu(I_n(z))\rho_q(z)\,dz \to \ell_\nu$. If $\ell_\nu > -\infty$, then $\ell_\nu \leq Q_\nu(I_n(z)) \leq Q_\nu(I_1(z))$ for every $n$ and $z$, and
\[
\int_\RR Q_\nu(I_1(z))\rho_q(z)\,dz = Q_\mu(p_1) < \infty.
\]
Thus dominated convergence yields $\int_\RR Q_\nu(I_n(z))\rho_q(z)\,dz \to \ell_\nu$. If instead $\ell_\nu = -\infty$, then $Q_\nu(I_1(z)) - Q_\nu(I_n(z)) \uparrow +\infty$ for every $z$, and monotone convergence applied to $\bigl(Q_\nu(I_1(z)) - Q_\nu(I_n(z))\bigr)\rho_q(z)$ gives again $\int_\RR Q_\nu(I_n(z))\rho_q(z)\,dz \to \ell_\nu$.
Since $\int_\RR Q_\nu(I_n(z))\rho_q(z)\,dz = Q_\mu(p_n)$ for every $n$, we obtain $Q_\mu(p_n) \to \ell_\nu$, hence $\ell_\mu = \ell_\nu$. Therefore, if $\ell_\mu \neq \ell_\nu$, then $Q_\alpha$ is bounded from below. 

Assume now that $\ell_\nu\in\RR$, that $Q_\nu$ is increasing on a neighborhood of $0$, and that $Q_\alpha$ is bounded from below. Fix $\varepsilon \in (0,\diam(\supp(q)))$. Since $Q_\alpha(u) \downarrow \ell_\alpha$ as $u \to 0^+$, there exists $\delta \in (0,1)$ such that $Q_\alpha(\delta) \leq \ell_\alpha + \varepsilon$. Hence, for every $p,v \in (0,\delta]$, we have $Q_\alpha(p) - Q_\alpha(v) \geq -\varepsilon$, and therefore
\[
\int_0^1 F_q(z+Q_\alpha(p)-Q_\alpha(v))\,dv \geq \int_0^\delta F_q(z-\varepsilon)\,dv = \delta F_q(z-\varepsilon).
\]
By monotonicity of $Q_\nu$,
\[
Q_\mu(p) \geq \int_\RR Q_\nu\bigl(\delta F_q(z-\varepsilon)\bigr)\rho_q(z)\,dz
\qquad\text{for every } p \in (0,\delta].
\]
Since $\varepsilon < \diam(\supp(q))$, the set $A := \{z \in \RR : F_q(z-\varepsilon) > 0\}$ has positive $q$-measure. Moreover, strict monotonicity of $Q_\nu$ near $0$ implies $Q_\nu(t) > \ell_\nu$ for every $t > 0$. Thus $Q_\nu\bigl(\delta F_q(z-\varepsilon)\bigr) \geq \ell_\nu$ for every $z$, with strict inequality on $A$, so
\[
\int_\RR Q_\nu\bigl(\delta F_q(z-\varepsilon)\bigr)\rho_q(z)\,dz > \ell_\nu.
\]
Letting $p \to 0^+$ in the previous inequality, and using $Q_\mu(p) \to \ell_\mu$, we get $\ell_\mu > \ell_\nu$.

The proof of~\ref{it:compact-support-fixed-point-2} is analogous. The final claim follows by combining the two endpoint statements.
\end{proof}

\begin{figure}
     \centering
          \begin{subfigure}{0.6\textwidth}
         \centering
         \includegraphics[width=\textwidth]{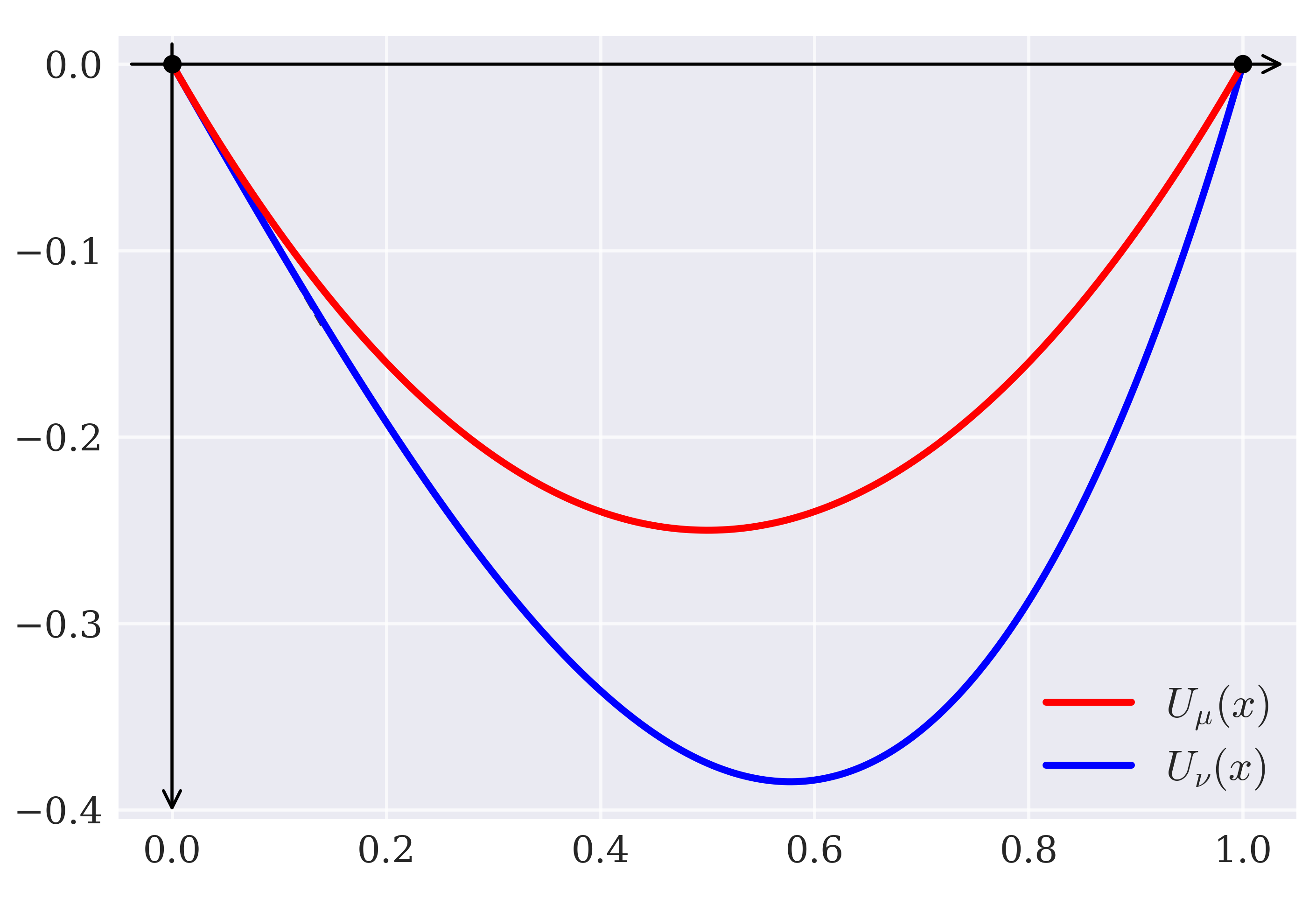}
     \end{subfigure}
	\caption{Here $\mu,\nu$ are such that $Q_\mu(u)=2u-1$, $Q_\nu(u)=3u^2-1$, and $(\mu,\nu)$ is irreducible. Although both $\mu$ and $\nu$ are compactly supported, with $\supp(\mu)=[-1,1]$ and $\supp(\nu)=[-1,2]$, Proposition~\ref{prop:compact-support-fixed-point} shows that any fixed-point distribution of a $q$-Bass martingale from $\mu$ to $\nu$ fails to be compactly supported, regardless of the choice of $q\in\Prob(\RR)$ with $q \ll \lambda$.}
     	\label{fig:non-compact-support-example}
\end{figure}

To obtain quantitative information on the support diameter from the data $\mu$, $\nu$, and $q$, we introduce the following family of $r$-maximal curves.

\begin{definition}[$r$-maximal curve]\label{def:r-maximal-curve}
	Let $\nu \in \Prob_1(\RR)$ and $q \in \Prob(\RR)$ be such that $q \ll \lambda$ and $\nu$ is not a Dirac measure. For any $r > 0$, the curve $\Gamma_r^{(\nu, q)}:[0,1] \rightarrow \RR$ defined by
	\[
		\Gamma_r^{(\nu, q)}(p) = \int_\RR Q_\nu\Big(pF_q(z)+(1-p)F_q(z-r)\Big)\, p \rho_q(z) \, dz
	\]
	is called the $r$-maximal curve associated with the $q$-Bass operator with terminal marginal $\nu$.
\end{definition}

\begin{proposition}[Lower bound on the support diameter]
\label{prop:maximal-curves-bound}
Let $r>0$, $\mu,\nu\in\Prob_1(\RR)$, and $q\in\Prob(\RR)$ be such that $q\ll\lambda$  and $\nu$ is not a Dirac measure.
Let $\Gamma_r$ be the $r$-maximal curve associated with the $q$-Bass operator with terminal marginal $\nu$.
Let $\alpha\in\Prob(\RR)$ be a fixed-point distribution of the $q$-Bass martingale from $\mu$ to $\nu$.
If there exists $p\in(0,1)$ such that
\[
	U_\mu(p)<\Gamma_r^{(\nu, q)}(p),
\]
then
\[
	\diam(\supp(\alpha))\ge r.
\]
\end{proposition}

\begin{figure}
     \centering
          \begin{subfigure}{0.6\textwidth}
         \centering
         \includegraphics[width=\textwidth]{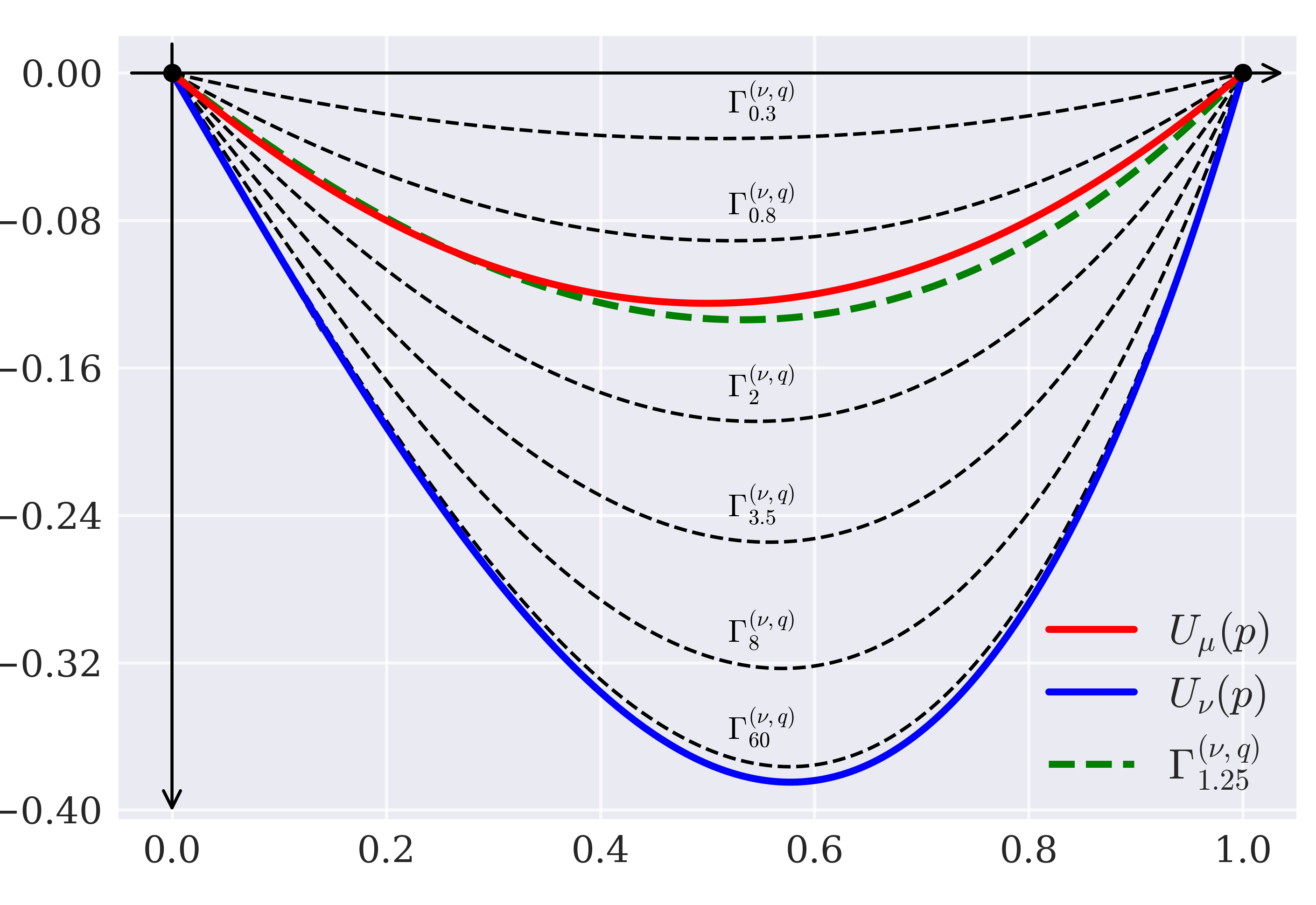}
     \end{subfigure}
	\caption{In this example, $Q_\mu(u)=u-\frac12$, $Q_\nu(u)=3u^2-1$, and $q$ is the standard Cauchy distribution. Since $U_\mu(0.2)<\Gamma_{1.25}^{(\nu,q)}(0.2)$, any fixed-point distribution of the $q$-Bass martingale from $\mu$ to $\nu$ must have support of diameter at least $1.25$, by Proposition~\ref{prop:maximal-curves-bound}.}
     	\label{fig:maximal-curves-example}
\end{figure}

\begin{proof}
We argue by contradiction. Fix $r>0$ and assume that $\diam(\supp(\alpha))<r$. Then $\alpha$ is compactly supported. In particular, we may assume that
\[
Q_\alpha \in L^\infty([0,1]) \cap \Mon((0,1), \RR)
\]
and that $J=\mathrm{ess}\sup_{u,v \in (0,1)} \bigl(Q_\alpha(u)-Q_\alpha(v)\bigr)=Q_\alpha(1)-Q_\alpha(0)\in [0,r)$.

For any $k \in \NN$, let
\[
I_k = \{i/k: i=0, \dots, k\} \cup \{p\}
\]
and set $m_k = |I_k|-1$. We denote by $0=u_0<u_1^{(k)}< \dots < u_{m_k}^{(k)}=1$ the elements of $I_k$, listed in increasing order. In particular, there exists $m_k^* \in \{1, \dots, m_k-1\}$ such that $u_{m_k^*}^{(k)} = p$.
Moreover, let $Q_k \in L^\infty([0,1]) \cap \Mon((0,1), \RR)$ be the piecewise constant function defined by
\[
Q_k(u)= Q_\alpha(u_i^{(k)}),\quad \text{if } u \in (u_i^{(k)}, u_{i+1}^{(k)}], \qquad \text{and} \qquad Q_k(0)=Q_k(u_{0}^{(k)}).
\]
Then $Q_k \to Q_\alpha$ pointwise. We define the vector $h^{(k)} \in \RR^{m_k-1}$ by
\[
h_i^{(k)}=\Delta Q_k(u_i^{(k)}), \qquad i =1, \dots, m_k-1.
\]
Since $Q_k$ is nondecreasing, we have $h^{(k)} \in \RR_{\geq 0}^{m_k-1}$ and $\|h^{(k)}\|_1 \leq J<r$.

Let $(\nu_k)_{k \in \NN}$ and $(q_k)_{k \in \NN}$ be the sequences given by Corollary~\ref{cor:approx_nu_2} and \cite[Proposition A.2]{AcMa26_semidiscrete}, respectively. We denote by
$\Mcal_k : \Mon((0,1), \RR) \rightarrow C([0,1])$
the $q_k$-Bass operator with terminal marginal $\nu_k$, and by
$f^{(k)}:\RR^{m_k-1}_{\geq 0} \rightarrow \RR^{m_k-1}$
the $m_k$-atomic $q_k$-Bass map with respect to the distribution $\nu_k$ and the weights
\[
p_i^{(k)} = u^{(k)}_{i+1}- u^{(k)}_{i}, \qquad i=0, \dots, m_k-1.
\]
Since $(\nu_k, q_k)$ satisfies Assumption~\ref{ass:technical-assumptions}, we can apply \cite[Proposition 2.9]{AcMa26_semidiscrete} to $f^{(k)}$ and obtain
\[
\Mcal_k(Q_k)(p)=f^{(k)}_{m_k^*}(h^{(k)}) \geq f^{(k)}_{m_k^*}(0, \dots, r, \dots, 0) = \Mcal_k(r\mathbf{1}_{(p, 1)}),
\]
where the first and last equalities follow from Remark~\ref{rmk:q-bass-operator-gen}, and $r$ appears in the $m_k^*$-th position of the vector.
Proposition~\ref{prop:q-Bass-operator-stability} then yields 
\[
\Mcal(Q_\alpha)(p) = \lim_{k \rightarrow \infty}\Mcal_k(Q_k)(p) \geq \lim_{k \rightarrow \infty} \Mcal_k(r \mathbf{1}_{(p, 1)}) = \Mcal(r \mathbf{1}_{(p, 1)}) = \Gamma_r^{(\nu, q)}(p),
\]
where $\Mcal$ is the $q$-Bass operator with terminal distribution $\nu$.
Since $\alpha$ is a fixed-point distribution, we have $U_\mu(p) = \Mcal(Q_\alpha)$. Therefore, $U_\mu(p) \geq \Gamma_r^{(\nu, q)}(p)$.
\end{proof}

We conclude this section by describing how fixed-point distributions, maximal curves, and the associated $q$-Bass coupling transform under positive affine transformations of the reference measure.

\begin{proposition}[$q$-invariance under linear transformations] \label{prop:scaling-properties}
	Let $r,a>0, b\in \RR$, $\mu,\nu\in\Prob_1(\RR)$, and $q\in\Prob(\RR)$ be such that $q\ll\lambda$ and $\nu$ is not a Dirac measure. Let us define the linear map $\ell:\RR \rightarrow \RR$ as $\ell(x)=ax+b$, for any $x\in \RR$. Then the following statements hold:
	\begin{enumerate}[label=(\roman*)]
		\item \label{it:scaling-properties-1} If $\alpha \in \Prob(\RR)$ is a fixed-point distribution for a $q$-Bass martingale $\pi \in \Mart(\mu, \nu)$, then $\pi$ is also an $\ell_\# q$-Bass martingale and $\ell_\# \alpha$ a fixed-point distribution for it.
		\item \label{it:scaling-properties-2} If $\Gamma_r^{(\nu, q)}$ is the $r$-maximal curve of the $q$-Bass operator with terminal marginal $\nu$, then 
		\[
			\Gamma_r^{(\nu, q)}(p) = \Gamma_{ar}^{\left(\nu, \ell_\#q\right)}(p), \qquad \text{for all } p \in [0,1].
		\]
	\end{enumerate}
	
\end{proposition}

\begin{proof}
	Let $\alpha \in \Prob(\RR)$ be a fixed-point solution of the $q$-Bass martingale from $\mu$ to $\nu$.  We first prove that $\ell_\#\alpha$ is a fixed point of the $\ell_\#q$-Bass martingale from $\mu$ to $\nu$. By \cite[Theorem 1.4]{AcMa26_semidiscrete}, this is equivalent to showing  that, for $p \in (0,1)$,
\[
	Q_\mu(p) = \int_\RR Q_\nu\!\left(\int_0^1 F_{\ell_\#q}(z+Q_{\ell_\#\alpha}(p)-Q_{\ell_\#\alpha}(v))\,dv\right)\rho_{\ell_\#q}(z)\,dz.
\]
We have
\begin{align*}
	\int_\RR Q_\nu & \!\left(\int_0^1 F_{\ell_\#q}(z+Q_{\ell_\#\alpha}(p)-Q_{\ell_\#\alpha}(v))\,dv\right)\rho_{\ell_\#q}(z)\,dz \\
	& = \int_\RR Q_\nu \!\left(\int_0^1 F_{\ell_\#q}\left(a \left(\frac{z}{a}+Q_\alpha(p)-Q_\alpha(v)\right)\right)\,dv\right)\rho_{\ell_\#q}(z)\,dz\\
	& = \int_\RR Q_\nu \!\left(\int_0^1 F_{q}\left(\frac{z-b}{a}+Q_\alpha(p)-Q_\alpha(v)\right)\,dv\right)\rho_{q}\left(\frac{z-b}{a}\right)\frac{1}{a}\,dz\\
	& = \int_\RR Q_\nu \!\left(\int_0^1 F_{q}\left(z+Q_\alpha(p)-Q_\alpha(v)\right)\,dv\right)\rho_{q}\left(z\right)\,dz = Q_\mu(p),
\end{align*}
where the last equality follows from \eqref{eq:fixed-point_equiv}.  Now recall that, by \cite[Proposition~2.1]{Ts24}, the optimization problems associated with the reference measures $q$ and $\ell_\#q$ each admits a unique optimizer, although the corresponding fixed points need not be unique. Thus, the $q$-Bass and $\ell_\#q$-Bass martingales are a priori distinct. Nevertheless, their transition kernels coincide and, since they have the same marginals, they define the same martingale coupling. Indeed, translation invariance implies that the coefficient $b$ has no effect on the transition kernel, while scaling both the fixed point $\alpha$ and the reference measure $q$ by a factor $a$ amounts to rescaling the argument of the transport map $T^{q}$ associated with the $q$-Bass martingale by the factor $1/a$. Consequently, the transport map $T^{\ell_\#q}$ associated with the $\ell_\#q$-Bass martingale satisfies
$T^{\ell_\#q}=T^{q}(\frac{\cdot}{a})$.

The proof of~\ref{it:scaling-properties-2} follows by applying the same argument to the definition of $r$-maximal curves.
\end{proof}

\section{Convergence of the Fixed-Point Iteration}
\label{sec:convergence}

The fixed-point characterization \eqref{eq:fixed_point_eq} of $q$-Bass martingales naturally suggests an iterative procedure for computing a Bass distribution: starting from a cumulative distribution function $F$, one repeatedly applies the operator $\Acal_q$ defined in \eqref{def:fixed_point_operator}. In this section, we study the convergence of this scheme in the compactly supported setting, which is the case most directly relevant for numerical applications.

We begin under additional regularity assumptions and extend to an arbitrary reference measure $q$ the differential approach developed in \cite{AcMaPa23} for the Gaussian case. Passing from distribution functions to quantile functions, we introduce an operator $\Gcal_q$ corresponding to $\Acal_q$ and study its Fr\'echet derivative. The resulting estimates show that, upon identifying cumulative distribution functions with their associated probability measures, $\Acal_q$ is non-expansive with respect to $\Wcal_\infty$. They further yield linear convergence both to the set of fixed points and to a fixed-point distribution. An approximation argument then allows us to retain the non-expansiveness property under substantially weaker assumptions.

We next establish convergence without imposing the regularity conditions required for a quantitative rate. Since non-expansiveness alone does not guarantee convergence of the iterates, we complement it with a strict descent property for the dual objective associated with the martingale transport problem. We show that this objective decreases strictly unless the current iterate is already a fixed point. Combining strict descent with compactness and non-expansiveness yields convergence of the fixed-point iteration. Furthermore, under the stronger assumptions considered in the first part of the section, this convergence is linear.

\subsection{Linear convergence under regularity assumptions}

We first establish linear convergence, with respect to the $\infty$-Wasserstein distance, of the fixed-point sequence
\[
  \bigl(\Acal_q^k F\bigr)_{k \in \NN} \subseteq \CDF,
\]
for any $F \in \CDF$, under Assumption~\ref{ass:technical-assumptions-2} and the assumptions below.

\begin{assumption}[Assumptions for linear $\Wcal_\infty$-convergence]\label{ass:linear-convergence}
Let $\mu, \nu, q \in \Prob(\RR)$. Assume that:
\begin{enumerate}[label=(A\arabic*)]
  \setcounter{enumi}{6}
  \item \label{ass:A7} the measure $\nu$ is absolutely continuous, and its density is bounded away from zero on $\co(\supp(\nu))$;
  \item \label{ass:A8} the measure $q$ is equivalent to the Lebesgue measure with square integrable density.
\end{enumerate}
\end{assumption}

To analyze $\Acal_q$ in the $\Wcal_\infty$ distance, it is convenient to reformulate the iteration on the quantile side. The auxiliary functions introduced below encode the transport step defining $\Acal_q$ and its behavior under perturbations of the input quantile function.

\begin{definition}
Let $\mu, \nu \in \Prob_\infty(\RR)$ and $q\in\Prob(\RR)$ be such that $q\ll\lambda$ and $\nu$ is not a Dirac measure. For $Q \in L^0(0,1)$, define
\begin{align}
  & S_Q^q(x) := \int_\RR Q_\nu\left( \int_0^1 F_q(x - Q(y) + z)\,dy \right)\rho_q(z)\,dz, & x \in \RR, \label{aux_operator_def1}\\
  & T_Q^q(x,y) := \int_\RR Q_\nu'\left(\int_0^1 F_q(x - Q(w) + z)\,dw \right)\rho_q(x-y+z)\rho_q(z)\,dz,& x,y \in \RR \nonumber.
\end{align}
Moreover, for $Q,f \in L^\infty(0,1)$, define
\begin{equation*}
  S_{Q,f}^q(\epsilon,x) := S_{Q+\epsilon f}^q(x),\quad \epsilon,x \in \RR.
\end{equation*}
\end{definition}
The main properties of these functions were proved for $q$ Gaussian in \cite{AcMaPa23}. Their generalization to any $q$ satisfying Assumption~\ref{ass:A8} is immediate.\footnote{In contrast to the Gaussian case, in the present setting one only has $S_Q^q \in C^1(\RR)$, which is sufficient for convergence. The Lipschitz estimate for $(S_Q^q)'$ requires additional regularity of $\rho_q$ and is not used here.}
Intuitively, the operator $Q \mapsto S_Q^q$ maps a random variable $Q \in L^0(0,1)$ with law $\alpha$ to $q \star T_{\alpha \ast q,\nu}$, where $T_{\alpha \ast q,\nu}$ denotes the unique monotone map from $\alpha \ast q$ to $\nu$. Moreover, the operator $(Q,f) \mapsto S_{Q,f}^q$ is introduced to study the sensitivity of $S_Q^q$ under perturbations of $Q$ of the form $Q+\epsilon f$, where $f \in L^\infty(0,1)$ and $\epsilon > 0$. 

To express the fixed-point iteration on the quantile side, we introduce the following counterpart of $\Acal_q$.

\begin{definition}
\label{def:G_q}
Let $\mu, \nu \in \Prob_\infty(\RR)$ and $q\in\Prob(\RR)$ be such that $q\ll\lambda$, $\nu$ is not a Dirac measure, and $\supp(\mu) \subseteq \interior(\co(\supp(\nu)))$. We define the operator
$\Gcal_q \colon L^\infty(0,1) \to L^\infty(0,1)$ by
\[
  \Gcal_q Q := (S_Q^q)^{-1} \circ Q_\mu,
\]
where $S_Q^q$ is given in \eqref{aux_operator_def1} and $(S_Q^q)^{-1}$ denotes its left-continuous generalized inverse.
\end{definition}

The following proposition collects the regularity and contraction properties of $\Gcal_q$ needed for the convergence analysis.

\begin{proposition}
\label{prop:properties-G_q}
Let $\mu, \nu \in \Prob_\infty(\RR)$ and $q\in\Prob(\RR)$ be such that $q\ll\lambda$, $\nu$ is not a Dirac measure, and $\supp(\mu) \subseteq \interior(\co(\supp(\nu)))$. Under Assumptions~\ref{ass:technical-assumptions-2} and
\ref{ass:linear-convergence}, the operator $\Gcal_q$ has the following
properties:
\begin{enumerate}[label=(\roman*)]
  \item \label{it:mainProposition.1}
  For every $Q \in L^\infty(0,1)$, the map $u \mapsto \Gcal_q Q(u)$ is left-continuous.

  \item \label{it:mainProposition.2}
  For every $F \in \CDF$ with quantile function $Q \in L^\infty(0,1)$, we have
  \[
    (\Acal_q F)^{-1} = \Gcal_q Q.
  \]

  \item \label{it:mainProposition.3}
  The operator $\Gcal_q$ is Fréchet differentiable, with derivative $D_Q \Gcal_q$.

  \item \label{it:mainProposition.4}
  For every $Q \in L^\infty(0,1)$, the operator norm of the Fréchet derivative $D_Q\Gcal_q$ is bounded by $1$, and $\|D_Q \Gcal_q(f)\|_\infty = \|f\|_\infty \iff f \in \mathrm{Span}(\mathbf 1)$.
\end{enumerate}
\end{proposition}

\begin{proof}
The proof follows the argument of \cite[Proposition~3.9]{AcMaPa23}, with the Gaussian distribution function and density replaced by $F_q$ and $\rho_q$, respectively, and with the corresponding versions of \cite[Lemmas~3.4, 3.6, and 3.7]{AcMaPa23}. In particular, $S_Q^q$ is continuously differentiable and increasing, with range $\operatorname{range}(S_Q^q)=\interior\bigl(\co(\supp(\nu))\bigr)$; see \cite[Lemma~3.4(i)--(ii)]{AcMaPa23}. The argument only uses the continuity and strict positivity of $(S_Q^q)'$.

The implicit-function theorem applied to the identity $S_{Q+\epsilon f}^q\bigl(\Gcal_q(Q+\epsilon f)(u)\bigr)=Q_\mu(u)$ yields
\[
D_Q\Gcal_q(f)(u)=\int_0^1f(v)K_Q^q(u,v)\,dv,
\]
where
\[
K_Q^q(u,v):=\frac{T_Q^q\bigl(\Gcal_qQ(u),Q(v)\bigr)}{(S_Q^q)'\bigl(\Gcal_qQ(u)\bigr)};
\]
see \cite[Proof of Proposition~3.9(iii), equations~(3.16)--(3.20)]{AcMaPa23}. The corresponding continuity properties of $T_Q^q$, $(S_Q^q)'$, and $\Gcal_qQ$, together with the uniform positive lower bound for $(S_Q^q)'$ on bounded subsets of $L^\infty(0,1)$, show that $Q\mapsto D_Q\Gcal_q$ is continuous in operator norm; see \cite[Proof of Proposition~3.9(iii), equations~(3.21)--(3.23)]{AcMaPa23}. Hence the operator above is indeed the Fréchet derivative of $\Gcal_q$.

For almost every $u\in(0,1)$, the map $v\mapsto K_Q^q(u,v)$ is strictly positive and satisfies $\int_0^1K_Q^q(u,v)\,dv=1$; see \cite[Proof of Proposition~3.9(iii), equations~(3.19)--(3.20)]{AcMaPa23}. Therefore,
\[
\bigl|D_Q\Gcal_q(f)(u)\bigr|\leq\int_0^1|f(v)|K_Q^q(u,v)\,dv\leq\|f\|_\infty,
\]
and thus $\|D_Q\Gcal_q\|\leq1$. To characterize the equality case, recall that, for each fixed $Q\in L^\infty(0,1)$, there exists $\varepsilon_Q>0$ such that $K_Q^q(u,v)\geq\varepsilon_Q$ for almost every $(u,v)\in(0,1)^2$; see \cite[Proof of Proposition~3.9(iv)]{AcMaPa23}. Let $M:=\|f\|_\infty$. If $f\notin\mathrm{Span}(\mathbf 1)$, then $\left|\int_0^1f(v)\,dv\right|<M$, and consequently, for almost every $u\in(0,1)$,
\[
\bigl|D_Q\Gcal_q(f)(u)\bigr|\leq\varepsilon_Q\left|\int_0^1f(v)\,dv\right|+\int_0^1|f(v)|\bigl(K_Q^q(u,v)-\varepsilon_Q\bigr)\,dv<M.
\]
Thus, $\|D_Q\Gcal_q(f)\|_\infty<\|f\|_\infty$ whenever $f$ is not constant. Conversely, the normalization of the kernel gives $D_Q\Gcal_q(c\mathbf 1)=c\mathbf 1$ for every $c\in\RR$, which proves the characterization of equality; see \cite[Proof of Proposition~3.9(iv)]{AcMaPa23}.
\end{proof}

The derivative estimate in Proposition~\ref{prop:properties-G_q} yields both the non-expansiveness of $\Acal_q$ and the linear convergence of its iterates.

\begin{theorem} \label{thm:linear-convergence}
Under Assumptions~\ref{ass:technical-assumptions-2} and
\ref{ass:linear-convergence}, let $(\mu,\nu)$ be irreducible and let $F \in \CDF$. Then $\Acal_q$ is non-expansive with respect to the $\infty$-Wasserstein distance,\footnote{By abuse of notation, we write $\Wcal_\infty(F_\xi,F_\zeta) := \Wcal_\infty(\xi,\zeta)$ for distributions $\xi,\zeta \in \Pcal(\RR)$.} that is,
\[
  \Wcal_\infty(\Acal_q F,\Acal_q G) \le \Wcal_\infty(F,G), \qquad F,G \in \CDF_c.
\]
Moreover, equality holds if and only if $F^{-1}-G^{-1} \in \mathrm{Span}(\mathbf 1) \subseteq L^\infty(0,1)$.

Furthermore, the sequence $(\Acal_q^k F)_{k \in \NN}$ converges to a fixed point
$F^\ast$ of $\Acal_q$, and there exists $\theta \in (0,1)$ such that we have:
\begin{enumerate}[label=(\roman*)]
  \item \label{it:conv_1}
  \emph{Linear convergence to the solution space:}
  for all $k \in \NN$,
  \[
    \Wcal_\infty(\Acal_q^{k+1}F,\Lcal^\ast)
    \le
    \theta \Wcal_\infty(\Acal_q^kF,\Lcal^\ast),
  \]
  where $\Lcal^\ast$ denotes the set of fixed points of $\Acal_q$.

  \item \label{it:conv_2}
  \emph{Convergence to a fixed point:}
  for all $k \in \NN$,
  \[
    \Wcal_\infty(\Acal_q^k F,F^\ast)
    \le
    2\theta^k \Wcal_\infty(F,F^\ast).
  \]
\end{enumerate}
\end{theorem}

\begin{proof}
The proof follows the arguments of \cite[Proofs of Theorems~1.2 and 1.4]{AcMaPa23}, given that Proposition~\ref{prop:properties-G_q} has been established. For $Q_1,Q_2\in L^\infty(0,1)$, the integral form of Taylor's formula gives
\[
\Gcal_qQ_1-\Gcal_qQ_2=\int_0^1D_{Q_2+t(Q_1-Q_2)}\Gcal_q(Q_1-Q_2)\,dt
\]
in $L^\infty(0,1)$. Proposition~\ref{prop:properties-G_q}\ref{it:mainProposition.4} then yields non-expansiveness, as in \cite[Proof of Theorem~1.2]{AcMaPa23}. If $Q_1-Q_2$ is not constant, then $\|D_{Q_2+t(Q_1-Q_2)}\Gcal_q(Q_1-Q_2)\|_\infty<\|Q_1-Q_2\|_\infty$ for every $t\in[0,1]$. Since the derivative depends continuously on its base point in operator norm, this estimate is uniform in $t$, and hence $\|\Gcal_qQ_1-\Gcal_qQ_2\|_\infty<\|Q_1-Q_2\|_\infty$. Conversely, if $Q_1-Q_2$ is constant, the translation invariance of $\Gcal_q$ yields equality. This proves the characterization of the equality case.

For linear convergence, the corresponding versions of \cite[Lemmas~3.4 and 3.7]{AcMaPa23} imply that, for every $R>0$, there exist $\varepsilon(R),\delta(R)>0$ such that $\varepsilon(R)\leq K_Q^q(u,v)\leq\delta(R)$ for every $Q$ with $\|Q\|_\infty\leq R$ and almost every $(u,v)\in(0,1)^2$; cf. \cite[Remark~3.11, equation~(3.24)]{AcMaPa23}. As in \cite[Proof of Theorem~1.4(i), before equation~(3.28)]{AcMaPa23}, non-expansiveness implies that the iterates and their closest fixed points remain in a common bounded subset of $L^\infty(0,1)$. Applying the integral form of Taylor's formula along the segment joining an iterate $Q_k$ to a closest fixed point $Q_k^\ast$ gives
\[
Q_k^\ast-Q_{k+1}=\int_0^1D_{Q_k+t(Q_k^\ast-Q_k)}\Gcal_q(Q_k^\ast-Q_k)\,dt
\]
in $L^\infty(0,1)$. The two-case estimate in \cite[Proof of Theorem~1.4(i), equations~(3.29)--(3.30) and the subsequent argument]{AcMaPa23} therefore applies without further modification and yields a constant $\theta\in(0,1)$ independent of $k$. Finally, the relative compactness of the bounded sequence of quantile functions, the lower semicontinuity of the $L^\infty$-norm, and the resulting estimate in \ref{it:conv_2} follow as in \cite[Proof of Theorem~1.4(ii)]{AcMaPa23}.
\end{proof}

\subsection{Non-expansiveness under weaker assumptions}

Although the linear convergence result in Theorem~\ref{thm:linear-convergence} relies on Assumption~\ref{ass:linear-convergence}, its non-expansiveness conclusion remains valid under considerably weaker hypotheses. This is obtained by approximating the reference measure and passing to the limit in the quantile representation of $\Acal_q$.

\begin{corollary}\label{cor:general-non-expansivness}
Let $\mu,\nu\in\Prob_\infty(\RR)$ and $q\in\Prob(\RR)$ be such that $q\ll\lambda$ and $\nu$ is not a Dirac measure. Let $F_1,F_2\in\CDF_c$, and suppose that Assumption~\ref{ass:technical-assumptions-2} holds. Assume, in addition, that one of the following conditions holds:
\begin{enumerate}[label=(\roman*)]
    \item \label{it:non-expan:assumptio1} $\supp(q)=\RR$; 
    \item \label{it:non-expan:assumptio2} $\supp(q)$ and $\supp(\mu)$ are intervals, $\mu,\nu\ll\lambda$, and $F_1,F_2$ are increasing on the convex hulls of their supports.
\end{enumerate}
Then
\[
\Wcal_\infty(\Acal_qF_1,\Acal_qF_2)\leq\Wcal_\infty(F_1,F_2).
\]
Under condition~\ref{it:non-expan:assumptio2}, $\Acal_qF_1$ and $\Acal_qF_2$ are also increasing on the convex hulls of their supports.

In particular, $\Wcal_\infty(\Acal_qF,F^*)\leq\Wcal_\infty(F,F^*)$ for every $F\in\CDF_c$ and every fixed point $F^*\in\CDF$ of $\Acal_q$, provided, under condition~\ref{it:non-expan:assumptio2}, that $F$ and $F^*$ are increasing on the convex hulls of their supports.
\end{corollary}

\begin{proof}
To emphasize the dependence on the marginals and the reference measure, write
\[
S_Q^{q,\nu}(x):=\int_\RR Q_\nu\left(\int_0^1F_q(x-Q(y)+z)\,dy\right)\rho_q(z)\,dz,\qquad \Gcal_q^{\mu,\nu}Q:=(S_Q^{q,\nu})^{-1}\circ Q_\mu,
\]
where $(S_Q^{q,\nu})^{-1}$ denotes the left-continuous inverse of $S_Q^{q,\nu}$. Set $Q_i:=F_i^{-1}$ for $i=1,2$. Let $(\mu_k,\nu_k)_{k\in\NN}$ be an approximating sequence as in \cite[Lemma~A.1]{AcMaPa23}, and let $(q_k)_{k\in\NN}$ be given by Proposition~\ref{prop:approx_q_2}. Thus, $(\mu_k,\nu_k,q_k)$ satisfies Assumptions~\ref{ass:technical-assumptions-2} and \ref{ass:linear-convergence} for every $k\in\NN$, while $\mu_k\to\mu$ and $\nu_k\to\nu$ in $\Wcal_1$ and $\rho_{q_k}\to\rho_q$ in $L^1(\RR)$. Hence, Theorem~\ref{thm:linear-convergence} gives
\begin{equation}\label{eq:non-expan-compact}
\left\|\Gcal_{q_k}^{\mu_k,\nu_k}Q_1-\Gcal_{q_k}^{\mu_k,\nu_k}Q_2\right\|_\infty\leq\|Q_1-Q_2\|_\infty,\qquad k\in\NN.
\end{equation}

We show that, under either condition~\ref{it:non-expan:assumptio1} or condition~\ref{it:non-expan:assumptio2},
\begin{equation}\label{eq:convergence-G-approx}
\Gcal_{q_k}^{\mu_k,\nu_k}Q_i(u)\longrightarrow\Gcal_q^{\mu,\nu}Q_i(u)
\end{equation}
for almost every $u\in(0,1)$ and $i=1,2$.

Assume first that condition~\ref{it:non-expan:assumptio1} holds. Since $\supp(q)=\RR$, the map $S_{Q_i}^{q,\nu}$ is continuous and increasing, with range $\interior(\co(\supp(\nu)))$. This follows as in \cite[Lemma~3.4(i)--(ii)]{AcMaPa23}, with the Gaussian distribution function and density replaced by $F_q$ and $\rho_q$. The convergence argument is the same as in \cite[Proof of Corollary~1.3]{AcMaPa23}, except that the reference measure is also approximated. To verify this additional point, let $U\sim\text{Unif}_{[0,1]}$ and denote by $\alpha_i$ the law of $Q_i(U)$. Then
\[
\left\|F_{\alpha_i\ast q_k}-F_{\alpha_i\ast q}\right\|_\infty\leq\left\|\rho_{q_k}-\rho_q\right\|_{L^1(\RR)}\longrightarrow0.
\]
Since $F_{\alpha_i\ast q}$ is increasing, the same dominated-convergence argument as in \cite[Proof of Corollary~1.3]{AcMaPa23} yields $S_{Q_i}^{q_k,\nu_k}(x)\to S_{Q_i}^{q,\nu}(x)$ for every $x\in\RR$. The convergence argument used there then gives
\[
(S_{Q_i}^{q_k,\nu_k})^{-1}\longrightarrow(S_{Q_i}^{q,\nu})^{-1}
\]
locally uniformly on $\interior(\co(\supp(\nu)))$. Together with $Q_{\mu_k}\to Q_\mu$ almost everywhere, this proves \eqref{eq:convergence-G-approx}.

Assume now that condition~\ref{it:non-expan:assumptio2} holds. Since $F_i$ is increasing on the convex hull of its support, $Q_i$ is continuous. Hence $\supp(\alpha_i)$ is an interval, and so is $\supp(\alpha_i\ast q)=\supp(\alpha_i)+\supp(q)$. Since $\alpha_i\ast q\ll\lambda$, the distribution function $F_{\alpha_i\ast q}$ is continuous and increasing on the interior of its support. Moreover, $\nu\ll\lambda$ implies that $Q_\nu$ is increasing. Therefore, $Q_\nu\circ F_{\alpha_i\ast q}$ is increasing on the interior of $\supp(\alpha_i\ast q)$. The convolution argument in \cite[Proof of Lemma~3.4(ii)]{AcMaPa23}, with the strict positivity of the Gaussian density replaced by the fact that every non-empty open subinterval of $\supp(q)$ has positive $q$-measure, shows that $S_{Q_i}^{q,\nu}$ is continuous and increasing on the region where it takes values in $\interior(\co(\supp(\nu)))$.

As above, $F_{\alpha_i\ast q_k}\to F_{\alpha_i\ast q}$ uniformly. Since $Q_\nu$ has at most countably many discontinuities and $F_{\alpha_i\ast q}$ is increasing on the interior of its support, for every $x\in\RR$ the argument of $Q_\nu$  belongs to its discontinuity set only on a $q$-null set of values of $z$. The same dominated-convergence argument therefore yields $S_{Q_i}^{q_k,\nu_k}(x)\to S_{Q_i}^{q,\nu}(x)$ for every $x\in\RR$. Since the limiting map is continuous and increasing on the relevant region,  the sequence of inverse functions $(S_{Q_i}^{q_k,\nu_k})^{-1}$ converges locally uniformly on $\interior(\co(\supp(\nu)))$, as in \cite[Proof of Corollary~1.3]{AcMaPa23}. Together with $Q_{\mu_k}\to Q_\mu$ almost everywhere, this proves \eqref{eq:convergence-G-approx} also in this case.

Using $(\Acal_qF_i)^{-1}=\Gcal_q^{\mu,\nu}Q_i$, the lower semicontinuity of the $L^\infty$-norm under almost-everywhere convergence and \eqref{eq:non-expan-compact} give
\[
\begin{aligned}
\Wcal_\infty(\Acal_qF_1,\Acal_qF_2)
&=\left\|\Gcal_q^{\mu,\nu}Q_1-\Gcal_q^{\mu,\nu}Q_2\right\|_\infty\leq\liminf_{k\to\infty}\left\|\Gcal_{q_k}^{\mu_k,\nu_k}Q_1-\Gcal_{q_k}^{\mu_k,\nu_k}Q_2\right\|_\infty\\
&\leq\|Q_1-Q_2\|_\infty
=\Wcal_\infty(F_1,F_2).
\end{aligned}
\]

Finally, under condition~\ref{it:non-expan:assumptio2}, the assumptions that $\mu\ll\lambda$ and that $\supp(\mu)$ is an interval imply that $Q_\mu$ is continuous. Since $(S_{Q_i}^{q,\nu})^{-1}$ is continuous on the range of $Q_\mu$, the quantile function
\[
(\Acal_qF_i)^{-1}=(S_{Q_i}^{q,\nu})^{-1}\circ Q_\mu
\]
is continuous. Hence $\Acal_qF_i$ is increasing on the convex hull of its support. The final claim follows by taking $F_2=F^*$, with the additional monotonicity assumptions stated under condition~\ref{it:non-expan:assumptio2}.
\end{proof}

\subsection{Convergence in the compactly supported setting}

Non-expansiveness controls the distance between different iterates but does not, by itself, guarantee convergence of the fixed-point sequence. To overcome this limitation under weaker assumptions, we combine it with a strict descent property of the dual objective along the iteration. In the Gaussian case, its strict descent along the fixed-point sequence was established in \cite[Theorem~3.5]{HaJoLoObPa26}. The next theorem extends this property to arbitrary absolutely continuous reference measures $q$. We restrict ourselves to compactly supported measures $\mu$ and $\nu$, which is the most relevant setting for numerical applications.

\begin{theorem}[Strict descent in the compact case]\label{thm:strict-descent}
Let $\mu,\nu\in\Prob_\infty(\RR)$ and let $q\in\Prob_1(\RR)$ be such that $q\ll\lambda$. Under Assumption~\ref{ass:technical-assumptions-2}, let $F\in\CDF$, and denote by $\alpha_k\in\Prob_\infty(\RR)$ the probability measure with CDF $\Acal_q^kF$. Let
\[
\Ecal^q(v):=\int_\RR v(x)\,\nu(dx)-\int_\RR(v^\ast\star q)^\ast(x)\,\mu(dx)
\]
be the dual objective associated with \eqref{eq:dual-problem}. For every $k\in\NN$, let $v_k$ be a convex lower semicontinuous function such that
\[
(v_k^\ast)'=Q_\nu\circ F_{\alpha_k\ast q}.
\]
Then, for every $k\geq1$,
\[
\Ecal^q(v_k)\leq\Ecal^q(v_{k-1}).
\]
Moreover, the inequality is strict unless $\alpha_k$ is a fixed-point distribution.
\end{theorem}

\begin{proof}
The proof is identical to that of \cite[Theorem~3.5]{HaJoLoObPa26}, with the Gaussian reference measure replaced by $q$.
\end{proof}

We can now combine non-expansiveness, compactness of the iterates, and strict descent of the dual objective to prove the main convergence result, that was stated in the Introduction.

\begin{proof}[Proof of Theorem~\ref{thm:convergence}]
Let $(Q_k)_{k\in\NN}$ and $(\alpha_k)_{k\in\NN}$ be defined as in the statement. Since $\alpha_0\in\Prob_\infty(\RR)$, we have $Q_k\in L^\infty(0,1)$ for every $k\in\NN$.
Fix $\overline\alpha\in\Lcal^\ast$ and denote its quantile function by $\overline Q$. By Proposition~\ref{prop:compact-support-fixed-point}, $\overline Q\in L^\infty(0,1)$. Since both $Q_0$ and $\overline Q$ are bounded, there exist $c_-,c_+\in\RR$ such that
\[
\overline Q+c_-\leq Q_0\leq\overline Q+c_+.
\]
By monotonicity and shift-invariance of $\Gcal_q$, the quantile functions $\overline Q+c_-$ and $\overline Q+c_+$ are fixed points and
\[
\overline Q+c_-\leq Q_k\leq\overline Q+c_+,\qquad k\in\NN.
\]
Consequently, the measures $(\alpha_k)_{k\in\NN}$ are supported in a common compact interval.

Let $v_k$ be the convex potential associated with $\alpha_k$ as in Theorem~\ref{thm:strict-descent}. If some $\alpha_k$ is a fixed-point distribution, then the sequence is constant from that index onward and there is nothing to prove. Otherwise, Theorem~\ref{thm:strict-descent} shows that $(\Ecal^q(v_k))_{k\in\NN}$ is decreasing.  Hence, $\Ecal^q(v_k)\longrightarrow L_{\Ecal}$ for some $L_{\Ecal}\in\RR$. In particular, $L_{\Ecal}>-\infty$. Indeed, otherwise the dual problem \eqref{eq:dual-problem} associated with \eqref{def:WT^q_S} would fail to have a finite value. This is impossible, since $\mu,\nu\in\Prob_\infty(\RR)\subseteq\Prob_2(\RR)$ and \eqref{def:WT^q_S} is finite, as
\[
|\Scal(\pi)|\leq \int_\RR \left|\MCov(\pi_x,q)\right|\, \mu(dx)\leq \frac{1}{2}\int_\RR |x|^2\,\mu(dx) + \frac{1}{2}\int_\RR |x|^2\,\nu(dx)<\infty.
\]

Set $\beta_k:=\alpha_k\ast q$, $T_k:=Q_\nu\circ F_{\beta_k}$, and $\varphi_k:=v_k^\ast$, where the additive constant in $\varphi_k$ is chosen so that $\varphi_k(0)=0$. Thus
\[
\varphi_k(x)=\int_0^xT_k(y)\,dy,\qquad x\in\RR.
\]
The quantile functions $Q_k$ are uniformly bounded, and the non-decreasing maps $T_k$ take values in the compact interval $\co(\supp(\nu))$. Therefore, by successive applications of Helly’s selection theorem, we may extract a subsequence $(k_n)_{n\in\NN}$ and find bounded quantile functions $\widetilde Q_-$, $\widetilde Q$, and $\widetilde Q_+$, together with non-decreasing functions $\widetilde T_-$ and $\widetilde T$, such that
\[
Q_{k_n-1}\longrightarrow\widetilde Q_-,\qquad 
Q_{k_n}\longrightarrow\widetilde Q, \qquad
Q_{k_n+1}\longrightarrow\widetilde Q_+, \qquad
T_{k_n-1}\longrightarrow\widetilde T_-,\quad \text{and} \quad
T_{k_n}\longrightarrow\widetilde T
\]
at every continuity point of the respective limit functions.
Denote by $\widetilde\alpha_-$, $\widetilde\alpha$, and $\widetilde\alpha_+$ the probability measures associated with $\widetilde Q_-$, $\widetilde Q$, and $\widetilde Q_+$. Since all these quantile functions take values in a common compact interval,
\[
\alpha_{k_n-1}\longrightarrow\widetilde\alpha_-,\qquad
\alpha_{k_n}\longrightarrow\widetilde\alpha,\qquad
\alpha_{k_n+1}\longrightarrow\widetilde\alpha_+
\]
in $\Wcal_1$. Consequently, $\beta_{k_n-1}\to\widetilde\beta_-:=\widetilde\alpha_-\ast q$ and $\beta_{k_n}\to\widetilde\beta:=\widetilde\alpha\ast q$ in $\Wcal_1$.

Since $q\ll\lambda$, the measures $\widetilde\beta_-$ and $\widetilde\beta$ are atomless. The identities
\[
T_k(Q_{\beta_k}(u))=Q_\nu(u),\qquad \lambda\text{-a.e. }u\in(0,1),
\]
together with the convergence of the quantiles and the monotonicity of the transport maps, show that $\widetilde T_-$ and $\widetilde T$ are the monotone transport maps from $\widetilde\beta_-$ and $\widetilde\beta$ to $\nu$, respectively. Additionally, setting
\[
\widetilde\varphi_-(x):=\int_0^x\widetilde T_-(y)\,dy,
\qquad
\widetilde\varphi(x):=\int_0^x\widetilde T(y)\,dy,
\]
we have $\varphi_{k_n-1}\to\widetilde\varphi_-$ and $\varphi_{k_n}\to\widetilde\varphi$ locally uniformly on $\RR$.

Set $\psi_k:=\varphi_k\star q$, $\widetilde\psi_-:=\widetilde\varphi_-\star q$, and $\widetilde\psi:=\widetilde\varphi\star q$. Since the functions $\varphi_k$ have a common Lipschitz constant and satisfy $\varphi_k(0)=0$, and $q\in\Prob_1(\RR)$, the above local uniform convergence implies $\psi_{k_n-1}\to\widetilde\psi_-$ and $\psi_{k_n}\to\widetilde\psi$ locally uniformly on $\RR$.
For every $k\in\NN$, the iterates satisfy
$Q_\mu(u)\in\partial\psi_k(Q_{k+1}(u))$ for $\lambda$-a.e. $u\in(0,1)$.  In particular, choosing $k=k_n-1$ and $k=k_n$, respectively, we obtain
$Q_\mu(u)\in\partial\psi_{k_n-1}(Q_{k_n}(u))$ and $Q_\mu(u)\in\partial\psi_{k_n}(Q_{k_n+1}(u))$. Passing to the limit as $n\to\infty$ in the two corresponding subgradient inequalities, associated with $k=k_n-1$ and $k=k_n$,
respectively, yields
\[
Q_\mu(u)\in\partial\widetilde\psi_-(\widetilde Q(u)),
\qquad
Q_\mu(u)\in\partial\widetilde\psi(\widetilde Q_+(u)),\qquad \text{for $\lambda$-a.e. $u\in(0,1)$.}
\]
Thus, $\widetilde\alpha$ is obtained from $\widetilde\alpha_-$, and $\widetilde\alpha_+$ from $\widetilde\alpha$, by applying the same fixed-point algorithm as in the original iteration.

We use the following elementary observation: if $\eta_n\to\eta$ in $\Wcal_1$, the functions $f_n$ have a common Lipschitz constant, $(f_n(0))_{n\in\NN}$ is bounded, and $f_n\to f$ locally uniformly, then
\[
\int_\RR f_n\,d\eta_n\longrightarrow\int_\RR f\,d\eta.
\]
Indeed, if $L$ is a common Lipschitz constant, then
\[
\left|\int_\RR f_n\,d\eta_n-\int_\RR f\,d\eta\right|
\leq L\Wcal_1(\eta_n,\eta)+\int_\RR|f_n-f|\,d\eta,
\]
and the last term converges to zero by dominated convergence.

Since $\varphi_k$ induces the monotone transport from $\beta_k$ to $\nu$, while $\psi_k^\ast$ induces the update from $\mu$ to $\alpha_{k+1}$, we have
\[
\Ecal^q(v_k)
=
\MCov(\beta_k,\nu)
-\int_\RR\varphi_k\,d\beta_k
-\MCov(\mu,\alpha_{k+1})
+\int_\RR\psi_k\,d\alpha_{k+1}.
\]
Moreover, for $\zeta\in\{\mu,\nu\}$, $\left|\MCov(\eta_n,\zeta)-\MCov(\eta,\zeta)\right|\leq\|Q_\zeta\|_\infty\Wcal_1(\eta_n,\eta)$. The above convergences therefore yield
\[
\Ecal^q(v_{k_n-1})\longrightarrow\Ecal^q(\widetilde\varphi_-^\ast),
\qquad
\Ecal^q(v_{k_n})\longrightarrow\Ecal^q(\widetilde\varphi^\ast).
\]
Since both subsequences of $(\Ecal^q(v_k))_{k\in\NN}$ converge to $L_{\Ecal}$,
\[
\Ecal^q(\widetilde\varphi^\ast)=\Ecal^q(\widetilde\varphi_-^\ast)=L_{\Ecal}.
\]
Applying the equality characterization in Theorem~\ref{thm:strict-descent} to the limiting pair $(\widetilde\alpha_-,\widetilde\alpha)$ shows that $\widetilde\alpha$ is a fixed-point distribution. Hence, the next application of the fixed-point algorithm leaves it unchanged, and therefore $\widetilde\alpha_+=\widetilde\alpha$.

Since the previous argument applies to every subsequence of $(\alpha_k)_{k\in\NN}$, each admits a further subsequence converging weakly to an element of $\Lcal^\ast$. Since all measures are supported on a common compact interval, weak convergence implies $\Wcal_p$-convergence for every $p<\infty$. Furthermore, if $\Wcal_p(\alpha_k,\Lcal^\ast)$ did not converge to zero, a subsequence bounded away from $\Lcal^\ast$ would admit a further subsequence converging in $\Wcal_p$ to an element of $\Lcal^\ast$, which is a contradiction. Therefore,
\[
\Wcal_p(\alpha_k,\Lcal^\ast)\longrightarrow0,\qquad p\in[1,\infty).
\]

Assume now that either condition~(i) or condition~(ii) holds. Let $\widetilde\alpha\in\Lcal^\ast$ be the limiting distribution obtained above and set $\widetilde S_-:=q\star\widetilde T_-$.
As in the proof of Corollary~\ref{cor:general-non-expansivness}, $S_{Q_{k_n-1}}^q\to\widetilde S_-$ locally uniformly, and the associated left-continuous generalized inverses converge uniformly on the compact range of $Q_\mu$. Since the limit functions satisfy $\widetilde Q=(\widetilde S_-)^{-1}\circ Q_\mu$, we obtain
\[
\Wcal_\infty(\alpha_{k_n},\widetilde\alpha)=\|Q_{k_n}-\widetilde Q\|_\infty=\left\|(S_{Q_{k_n-1}}^q)^{-1}\circ Q_\mu-(\widetilde S_-)^{-1}\circ Q_\mu\right\|_\infty\longrightarrow0.
\]
Under condition~(ii), the distribution function of $\widetilde\alpha$ is increasing on the convex hull of its support by Lemma~\ref{lem:fixed-point-connected-support}. Hence, by
Corollary~\ref{cor:general-non-expansivness},
\[
\Wcal_\infty(\alpha_{k+1},\widetilde\alpha) \leq \Wcal_\infty(\alpha_k,\widetilde\alpha),\qquad k\in\NN.
\]
Thus $(\Wcal_\infty(\alpha_k,\widetilde\alpha))_{k\in\NN}$ is non-increasing and has a subsequence converging to zero. Hence $\Wcal_\infty(\alpha_k,\widetilde\alpha)\to0$, and therefore $\Wcal_\infty(\alpha_k,\Lcal^\ast)\to0$.

Finally, if Assumption~\ref{ass:linear-convergence} holds, Theorem~\ref{thm:linear-convergence} yields  $\theta\in(0,1)$ such that
\[
\Wcal_\infty(\alpha_{k+1},\Lcal^\ast)\leq\theta\,\Wcal_\infty(\alpha_k,\Lcal^\ast),\qquad\text{and}\qquad\Wcal_\infty(\alpha_{k+1},\widetilde\alpha)\leq2\theta^k\Wcal_\infty(\alpha_1,\widetilde\alpha),\qquad k\in\NN.
\]
\end{proof}

\begin{lemma}\label{lem:fixed-point-connected-support}
Let $\mu,\nu\in\Prob_\infty(\RR)$, and $q\in\Prob(\RR)$ satisfy $q\ll\lambda$. Assume that $\supp(q)$ and $\supp(\mu)$ are intervals and that $\mu,\nu\ll\lambda$. Then every fixed point $F^\ast$ of $\Acal_q$ is increasing on the convex hull of its support.
\end{lemma}

\begin{proof}
Let $\alpha^\ast$ be the probability measure with CDF $F^\ast$, set $Q^\ast:=(F^\ast)^{-1}$, and let $T:=Q_\nu\circ F_{\alpha^\ast\ast q}$. Then $S_{Q^\ast}^q=q\star T$ is continuous. Since $F^\ast$ is a fixed point, $F^\ast=F_\mu\circ S_{Q^\ast}^q$. In particular, $F^\ast$ is continuous because both $F_\mu$ and $S_{Q^\ast}^q$ are continuous.

Suppose, by contradiction, that $\supp(\alpha^\ast)$ is not an interval. Then there exist $a<b$ in $\supp(\alpha^\ast)$ such that $\alpha^\ast((a,b))=0$. Since $F^\ast$ is continuous, it is constant on $[a,b]$. Moreover, $F_\mu$ is increasing on $\supp(\mu)$ because $\supp(\mu)$ is an interval. The fixed-point identity therefore implies that $S_{Q^\ast}^q$ is constant on $[a,b]$.
We show that this is impossible. Let $X\sim\alpha^\ast$ and let $Z,Z'$ be independent random variables with law $q$, independent of $X$. Since $\supp(q)$ is an interval and $q\ll\lambda$, the support of $\Law(Z-Z')$ contains a neighborhood of zero. Choosing $\varepsilon\in(0,(b-a)/3)$ sufficiently small, we have
\[
\mathbb P\bigl(X\in(a-\varepsilon,a],\ \varepsilon<Z-Z'<2\varepsilon\bigr)>0.
\]
Consequently,
\[
\mathbb P\bigl(a<X+Z-Z'<b\bigr)>0,
\]
and hence
\[
\int_\RR\left[F_{\alpha^\ast\ast q}(b+z)-F_{\alpha^\ast\ast q}(a+z)\right]q(dz)>0.
\]
Thus, the set of points $z$ for which $F_{\alpha^\ast\ast q}(b+z)>F_{\alpha^\ast\ast q}(a+z)$ has positive $q$-measure. Since $\nu$ is atomless, $Q_\nu$ is increasing, and therefore
\[
S_{Q^\ast}^q(b)-S_{Q^\ast}^q(a)
=
\int_\RR\left[T(b+z)-T(a+z)\right]q(dz)>0.
\]
This contradicts the fact that $S_{Q^\ast}^q$ is constant on $[a,b]$. Hence $\supp(\alpha^\ast)$ is an interval, which is equivalent to $F^\ast$ being increasing on the convex hull of its support.
\end{proof}

\section{The Dynamic Problem}
\label{sec:dynamic}
The $q$-Bass martingale considered so far is a two-step martingale coupling between prescribed initial and terminal distributions. This naturally raises the question of whether the same construction admits a canonical dynamic interpolation and, if so, whether the resulting process can be characterized as the optimizer of a suitable dynamic martingale optimal transport problem.
To address these questions, we replace the reference random variable with law $q$ by an additive process $X$, whose independent increments provide reference dynamics at intermediate times. Although $X$ need not be a martingale, its semimartingale decomposition identifies a continuous local martingale part $X^c$ and a compensated jump measure $\widetilde\Ncal$. Moreover, in its natural filtration, an additive process has the weak representation property: every local martingale can be represented in terms of $X^c$ and $\widetilde\Ncal$.
This representation provides the natural framework for formulating dynamic counterparts of the static problems \eqref{def:WT^q_S} and \eqref{def:WT^q_I}. We will show that, whenever an $X$-Bass martingale in the sense of Definition~\ref{def:X_Bass} exists, it solves both dynamic optimization problems.
When $X$ is a standard Brownian motion, the $X$-Bass martingale coincides with the stretched Brownian motion, and the resulting optimization problem is precisely the martingale Benamou--Brenier formulation in \cite{BaBeHuKa20}.

\begin{definition}[Additive process]
Let $(\Omega,\Fcal,\mathbb F,\PP)$ be a filtered probability space. A real-valued process $X=(X_t)_{t\geq0}$ is called an additive process if
 $X_0=0$ a.s., $X$ has c\`adl\`ag paths and independent increments, and $X$ is stochastically continuous.
\end{definition}
Unlike a L\'evy process, an additive process need not have stationary increments. It therefore provides a time-dependent reference dynamics while retaining the independent-increment structure underlying the $q$-Bass construction.
We denote by $\mathbb F^X=(\mathcal F_t^X)_{t\in[0,1]}$  the completed, right-continuous natural filtration of $X$.

\begin{definition}[$X$-Bass martingale]\label{def:X_Bass}
Let $\mu,\nu\in\Prob_1(\RR)$, let $X=(X_t)_{t\in[0,1]}$ be an additive process, and set $q:=\Law(X_1)$. A martingale $M=(M_t)_{t\in[0,1]}$ is called an $X$-Bass martingale from $\mu$ to $\nu$ if $M_0\sim\mu, M_1\sim\nu$, and there exist a probability measure $\alpha\in\Prob(\RR)$, a non-decreasing map $T:\RR\to\RR$, and a random variable $A\sim\alpha$ independent of $X$, such that $T_\#(\alpha\ast q)=\nu$, $q\star T$ is increasing on a set of full $\alpha$-measure, and
\[
M_t=\EE\bigl[T(A+X_1)\mid\sigma(A)\vee\mathcal F_t^X\bigr],,\qquad t\in[0,1],
\]
or equivalently, setting $q_{t,1}:=\Law(X_1-X_t)$, 
\[
M_t=(q_{t,1}\star T)(A+X_t),\qquad t\in[0,1].
\]
\end{definition}

\begin{remark}
Suppose that $X_1\sim q\in\Prob(\RR)$ with $q\ll\lambda$. Then $\alpha\ast q$ is absolutely continuous and the non-decreasing map $T$ in the preceding definition coincides $(\alpha\ast q)$-a.s.\ with
\[
T=Q_\nu\circ F_{\alpha\ast q}.
\]
Moreover,
\[
M_0=(q\star T)(A),\qquad M_1=T(A+X_1).
\]
Consequently, $\Law(M_0,M_1)$ is a $q$-Bass martingale coupling from $\mu$ to $\nu$. 
\end{remark}

\begin{remark}[Discrete-time interpolation]
Let $Y_1,\dots,Y_n$ be independent random variables, and set $S_0:=0$ and $S_k:=\sum_{j=1}^kY_j$ for $k=1,\dots,n$. Fix  initial and terminal marginals $\mu,\nu\in\Prob_1(\RR)$, and let $\alpha$ be a fixed-point distribution for the $\Law(S_n)$-Bass martingale from $\mu$ to $\nu$. If $A\sim\alpha$ is independent of $(Y_1,\dots,Y_n)$ and $T$ is the corresponding non-decreasing terminal transport map from $\Law(S_n)$ to $\nu$, then
\[
M_k:=\EE\bigl[T(A+S_n)\mid A,S_k\bigr],\qquad k=0,\dots,n,
\]
defines a discrete-time martingale satisfying $M_0\sim\mu$ and $M_n\sim\nu$. In particular, the coupling $\Law(M_0,M_n)$ is a $\Law(S_n)$-Bass martingale coupling from $\mu$ to $\nu$.

More generally, for $k=0,\dots,n$, set
\[
    T_k(x):=\EE\bigl[T(x+S_n-S_k)\bigr],
\]
so that $M_k=T_k(A+S_k)$. If $0\leq r<s\leq n$ and $q_{r,s}:=\Law(S_s-S_r)$, then
\[
T_r=q_{r,s}\star T_s
\]
and the coupling $\Law(M_r,M_s)$ admits the representation
\[
M_r=(q_{r,s}\star T_s)(A+S_r),\qquad M_s=T_s(A+S_r+S_s-S_r).
\]
It is therefore a $q_{r,s}$-Bass martingale coupling from $\mu_r:=\Law(M_r)$ to $\mu_s:=\Law(M_s)$ whenever $T_r=q_{r,s}\star T_s$ is increasing on a set of full $\alpha_r$-measure. This condition holds, for instance, if $q_{r,s}\sim\lambda$ and $T_s$ is not constant.
\end{remark}

\begin{figure}
    \centering
    \includegraphics[width=\linewidth]{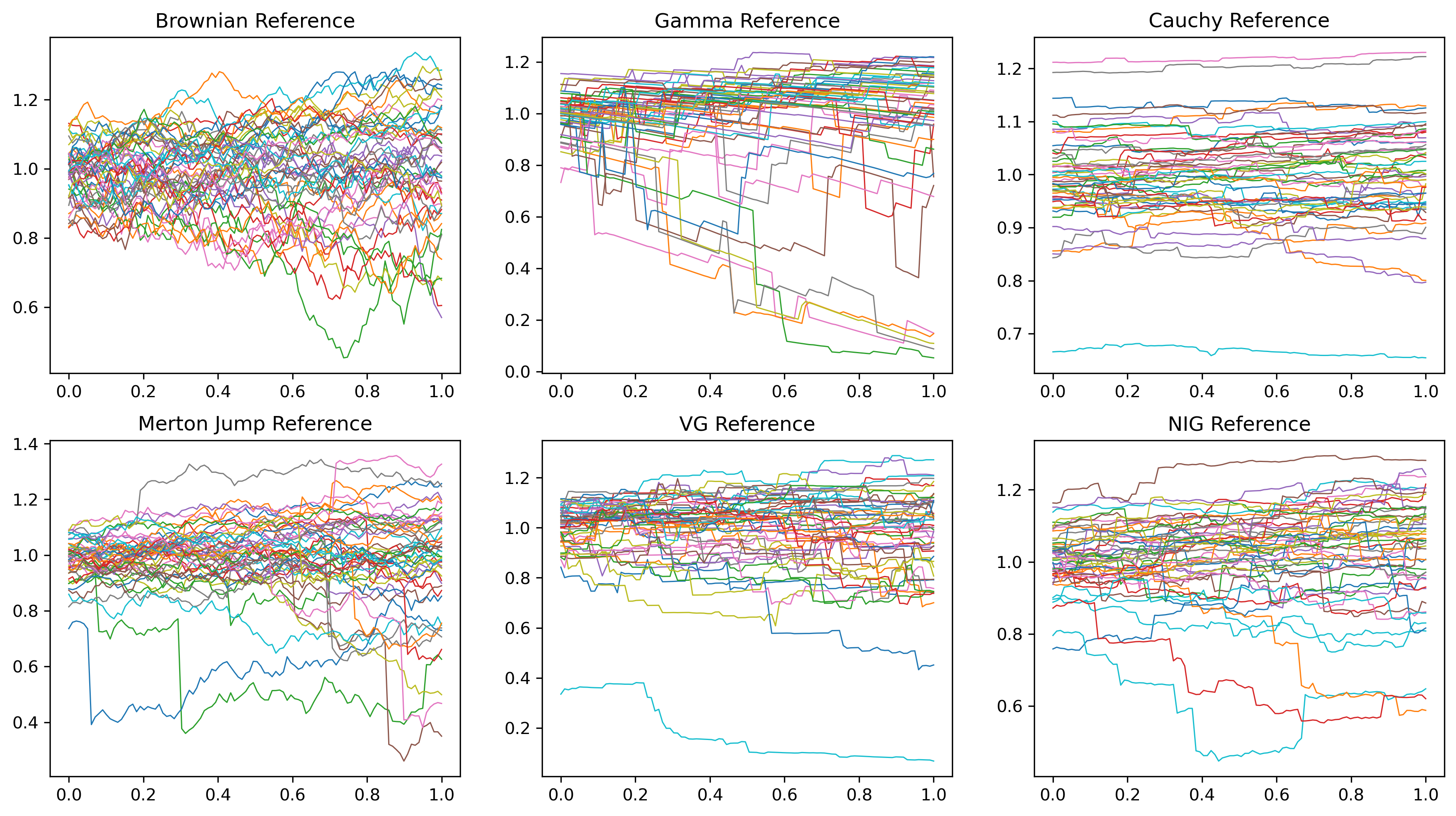}
    \caption{Illustration of the $X$-Bass martingales between the same prescribed initial and terminal distributions for different choices of the reference process $X$.}
    \label{fig:X-bass}
\end{figure}

To identify the dynamic optimization problem solved by an $X$-Bass martingale, we next recall the semimartingale structure of additive processes.

\begin{remark}[Canonical decomposition]\label{rem.c_dec}
Let $X$ be an additive process, and denote by $\Ncal$ its jump measure, by $J$ the compensator of $\Ncal$, and by
$\widetilde\Ncal:=\Ncal-J$ the compensated jump measure.
Fix the truncation function
$h(x):=x\mathbf 1_{\{|x|\leq1\}}$. Then there exist a deterministic càdlàg function $A$ of finite variation, with $A_0=0$, and a deterministic continuous increasing function $C$, with $C_0=0$, such
that
\[
X_t=X_0+A_t+X_t^c
+\int_{(0,t]\times\RR}h(x)\,\widetilde\Ncal(ds,dx)
+\int_{(0,t]\times\RR}(x-h(x))\,\Ncal(ds,dx),
\]
where $X^c$ is the continuous local martingale part of $X$ and satisfies
$\langle X^c\rangle_t=C_t$.
\end{remark}

The $X$-Bass construction includes an independent initial randomization. We therefore work with an initial enlargement of the natural filtration of $X$. Let $U$ be a uniform in $[0,1]$ random variable independent of $X$, and let $\mathbb G=(\mathcal G_t)_{t\in[0,1]}$ be the usual augmentation of
\[
\mathcal G_t:=\sigma(U)\vee\mathcal F_t^X,\qquad t\in[0,1].
\]
Since $U$ is independent of $\mathcal F_1^X$, the process $X$ has $\mathbb G$-independent increments and is therefore an additive process relative to $\mathbb G$.

Since an additive martingale has deterministic characteristics, its continuous martingale part has deterministic quadratic variation and its jump measure has deterministic compensator. Hence, \cite[Theorem~8 and Lemma~3(ii)]{Co13} yields the following martingale representation theorem.

\begin{proposition}
\label{prop:additive-martingale-representation}
Every $\mathbb G$-local martingale $M$ admits a representation of the form
\begin{equation}\label{eq.repGlm}
M_t=M_0+\int_0^tH_s\,dX_s^c+\int_{(0,t]\times\RR}W(s,z)\,\widetilde\Ncal(ds,dz),
\end{equation}
for suitable  $\mathbb G$-predictable integrands $H$ and $W$. Conversely, every process of this form is a $\mathbb G$-local martingale whenever the stochastic integrals are well defined.
\end{proposition}

From now on, we assume that $X$ is centered at every time and square integrable, namely, in terms of the canonical decomposition in Remark~\ref{rem.c_dec},
\begin{equation}\label{eq.Xsqint}
\EE[X_t]=0\; \;\forall\ t\in[0,1],\quad \text{and}\quad
C_1+\int_{(0,1]\times\RR}|z|^2\,J(ds,dz)<\infty.
\end{equation}
Then $X$ is a square-integrable additive martingale and
\[
X_t=X_t^c+\int_{(0,t]\times\RR}z\,\widetilde\Ncal(ds,dz),\qquad t\in[0,1].
\]

\subsection{Static-dynamic equivalence}
For $\mu,\nu\in\Prob_2(\RR)$, let $\mathfrak M_X(\mu,\nu)$ denote the class of square-integrable $\mathbb G$-martingales $M=(M_t)_{t\in[0,1]}$ such that $M_0\sim\mu$ and $M_1\sim\nu$. By Proposition~\ref{prop:additive-martingale-representation}, every $M\in\mathfrak M_X(\mu,\nu)$ admits representation \eqref{eq.repGlm}, for suitable square-integrable predictable integrands $H$ and $W$.
The martingale part of the reference process corresponds to the integrands $H_s=1$ and $W(s,z)=z$. This suggests measuring the deviation of an admissible martingale from the reference dynamics via
\[
\Ical_X(M):=\EE\left[\int_0^1|H_s-1|^2\,dC_s+\int_{(0,1]\times\RR}|W(s,z)-z|^2\,J(ds,dz)\right].
\]
Setting $q:=\Law(X_1)$, we define the dynamic counterpart of the static problem \eqref{def:WT^q_I} by
\begin{equation}
\label{eq:dyn-OT-1}
\inf_{M\in\mathfrak M_X(\mu,\nu)}\Ical_X(M).
\end{equation}
The two terms in $\Ical_X$ separately compare the continuous and jump components of the martingale dynamics with the corresponding components of the reference process. Thus, \eqref{eq:dyn-OT-1} seeks, among the admissible martingales with the prescribed marginals, those whose martingale dynamics are closest to those generated by $X$.

As in the static formulation, the corresponding correlation problem is obtained by expanding the quadratic cost. For $M\in\mathfrak M_X(\mu,\nu)$ with representation \eqref{eq.repGlm}, define
\[
\Scal_X(M):=\EE\left[\int_0^1H_s\,dC_s+\int_{(0,1]\times\RR}W(s,z)z\,J(ds,dz)\right].
\]
We will show that the dynamic counterpart of \eqref{def:WT^q_S} is
\begin{equation}
\label{eq:dyn-OT-2}
\sup_{M\in\mathfrak M_X(\mu,\nu)}\Scal_X(M).
\end{equation}
More precisely, the following proposition establishes the equivalence between the dynamic problems \eqref{eq:dyn-OT-1} and \eqref{eq:dyn-OT-2} and their static counterparts \eqref{def:WT^q_I} and \eqref{def:WT^q_S}, respectively, and shows that, whenever it exists, the $X$-Bass martingale is optimal for both formulations.

\begin{proposition}[Static--dynamic equivalence and optimality of the $X$-Bass martingale]
\label{prop:dynamic-optimality-X-Bass}
Let $\mu,\nu\in\Prob_2(\RR)$ satisfy $\mu\preceq_c\nu$, and $X=(X_t)_{t\in[0,1]}$ be an additive martingale satisfying \eqref{eq.Xsqint}.
Set $q:=\Law(X_1)$. Then
\[
\inf_{M\in\mathfrak M_X(\mu,\nu)}\Ical_X(M)=WT_I^q(\mu,\nu)-\int_\RR|x|^2\,\mu(dx)
\]
and
\[
\sup_{M\in\mathfrak M_X(\mu,\nu)}\Scal_X(M)=WT_S^q(\mu,\nu).
\]
More precisely, the endpoints coupling of every optimizer of either dynamic problem is an optimizer of the corresponding static problem, and every optimizer of either static problem admits a lift to an optimizer in $\mathfrak M_X(\mu,\nu)$. In particular, the dynamic problems \eqref{eq:dyn-OT-1} and \eqref{eq:dyn-OT-2} have the same optimizers.

Moreover, if there exists an $X$-Bass martingale $M^X\in\mathfrak M_X(\mu,\nu)$, then
\[
M^X\in\operatorname*{argmin}_{M\in\mathfrak M_X(\mu,\nu)}\Ical_X(M)\qquad\text{and}\qquad M^X\in\operatorname*{argmax}_{M\in\mathfrak M_X(\mu,\nu)}\Scal_X(M).
\]
\end{proposition}

\begin{proof}
Let $M\in\mathfrak M_X(\mu,\nu)$, set $\pi:=\Law(M_0,M_1)$, and let $(H,W)$ be its representation integrands in \eqref{eq.repGlm}. Since $X$ is a square-integrable additive martingale,
\[
X_t=X_t^c+\int_{(0,t]\times\RR}z\,\widetilde\Ncal(ds,dz).
\]
The It\^o isometry and the orthogonality of the continuous and purely discontinuous martingale components yield
\[
\Ical_X(M)=\EE\left[\left|(M_1-M_0)-X_1\right|^2\right].
\]
Since $M$ is a martingale and $M_0$ is $\mathcal G_0$-measurable, while $X$ is independent of $\mathcal G_0$, we have
\[
\EE[M_0(M_1-M_0)]=0,\qquad \EE[M_0X_1]=0.
\]
Consequently,
\[
\Ical_X(M)=\EE\left[|M_1-X_1|^2\right]-\int_\RR|x|^2\,\mu(dx).
\]
Similarly,
\[
\Scal_X(M)=\EE\left[(M_1-M_0)X_1\right]=\EE[M_1X_1].
\]
Conditionally on $M_0=x$, the joint law of $(M_1,X_1)$ is a coupling of $\pi_x$ and $q$. Therefore,
\[
\Ical_X(M)+\int_\RR|x|^2\,\mu(dx)=\int_\RR\EE\left[|M_1-X_1|^2\mid M_0=x\right]\mu(dx)\geq\int_\RR\Wcal_2^2(\pi_x,q)\,\mu(dx)=\Ical(\pi),
\]
and
\[
\Scal_X(M)\leq\int_\RR\sup_{\kappa\in\Pi(\pi_x,q)}\int_{\RR^2}yz\,\kappa(dy,dz)\,\mu(dx)=\Scal(\pi).
\]
Taking the infimum and the supremum gives one direction of the claimed identities.

Conversely, let $\pi\in\Mart(\mu,\nu)$. For $\mu$-a.e.\ $x$, choose an optimal coupling $\kappa_x\in\Pi(\pi_x,q)$ for the quadratic transport cost, equivalently for the covariance maximization problem. Using the independent randomization contained in $\mathcal G_0$, we may construct random variables $M_0$ and $M_1$ such that $M_0\sim\mu$ and
\[
\Law(M_1,X_1\mid M_0=x)=\kappa_x
\]
for $\mu$-a.e.\ $x$. Since the first marginal of $\kappa_x$ is $\pi_x$, and $\pi$ is a martingale coupling,
\[
\EE[M_1\mid M_0]=M_0.
\]
Define
\[
M_t:=\EE[M_1\mid\mathcal G_t],\qquad t\in[0,1].
\]
Then $M\in\mathfrak M_X(\mu,\nu)$, and Proposition~\ref{prop:additive-martingale-representation} gives its representation in terms of $X^c$ and $\widetilde\Ncal$. By construction, the preceding inequalities are equalities. Taking $\pi$ to be optimal proves both value identities and the correspondence between static and dynamic optimizers.

Finally, suppose that $M^X$ is an $X$-Bass martingale generated by $(\alpha,T,A)$. Then
\[
M_0^X=(q\star T)(A),\qquad M_1^X=T(A+X_1).
\]
Since $q\star T$ is increasing on a set of full $\alpha$-measure, $A$ is, up to null sets, a measurable function of $M_0^X$. Hence, conditionally on $M_0^X$, the coupling of $M_1^X$ and $X_1$ is induced by the non-decreasing map $z\mapsto T(A+z)$. It is therefore the monotone optimal coupling between the conditional law of $M_1^X$ and $q$. Thus equality holds in both of the preceding bounds, and $M^X$ is optimal for \eqref{eq:dyn-OT-1} and \eqref{eq:dyn-OT-2}.
\end{proof}

\begin{remark}[Reference processes with atoms]
The definition of an $X$-Bass martingale does not intrinsically require $q=\Law(X_1)$ to be absolutely continuous. Absolute continuity ensures that $\alpha\ast q$ is atomless and that the monotone terminal transport may be written as
\[
T=Q_\nu\circ F_{\alpha\ast q}.
\]
When $\alpha\ast q$ has atoms, a deterministic non-decreasing map $T$ satisfying $T_\#(\alpha\ast q)=\nu$ need not exist. The construction nevertheless remains possible whenever the atomic structures of $\alpha\ast q$ and $\nu$ are compatible with such a transport.

An example arises in the dynamic reinsurance problem studied in \cite{AcGaMaPa26}, where the reference process is a compound Poisson claims process. Its terminal distribution has an atom corresponding to the event that no claim occurs. When the prescribed terminal distribution contains a compatible atom, the terminal transport map can still be constructed, and hence so can the corresponding $X$-Bass martingale. In that setting, martingale optimal transport is used to define dynamic reinsurance strategies subject to terminal-distribution, moment, or risk-based constraints.
\end{remark}

\begin{remark}[On a geometric extension]
The geometric martingale Benamou--Brenier problems studied in \cite{BaLoOb25,BePaRi25} identify a positive martingale whose stochastic logarithm is closest to Brownian motion. Their connection with the classical martingale Benamou--Brenier problem is obtained through a change of num\'eraire and Girsanov's theorem.
A crucial ingredient is the rigidity of Brownian motion under an equivalent change of measure. The quadratic variation of a continuous semimartingale is invariant under such a change, and L\'evy's characterization identifies a continuous local martingale with quadratic variation $t$ as Brownian motion under the new measure.
The same argument does not extend directly to additive processes with jumps. Under an equivalent change of measure, Girsanov's theorem generally modifies the compensator of the jump measure and therefore changes the jump characteristics of the reference process. The jump component of a L\'evy process is thus not preserved with the same rigidity as the quadratic variation of Brownian motion. Consequently, the change-of-num\'eraire technique  does not directly produce an analogous geometric formulation for exponential L\'evy processes.
\end{remark}

\subsection{The $X$-Bass Martingale as Adapted Wasserstein Projection of the Reference Process $X$}
\label{sec:aw-projection}
We finally show that the $X$-Bass martingale can be characterized as the adapted Wasserstein projection of the additive process $X$ onto the class of square-integrable $\mathbb G$-martingales with marginals $\mu$ and $\nu$.
Adapted Wasserstein distances provide a variant of the classical Wasserstein distances by restricting the set of admissible couplings to those that preserve the underlying flow of information. Originating from the nested distance in discrete-time stochastic optimization \cite{PfPi12} and the subsequent development of causal and bicausal optimal transport \cite{La18, BaBeLiZa17}, this framework has also proved useful in continuous time \cite{AcBaZa20,BaBaBeEd19a}, in particular for stability and projection problems for stochastic processes \cite{BaBeHuKa20, BaKaRo22}.
To formulate the projection problem relevant here, we use the path-space approach of \cite[Section~6]{BaBeHuKa20}, adapted to processes that may have jumps. Therefore, we take as canonical space the Skorokhod space
\[
\Omega:=D([0,1],\RR).
\]
The idea is to view the law of the reference additive process $X$ as a probability measure on $\Omega$ and to seek, among all martingale laws with prescribed initial and terminal marginals, the one that is closest to it in an adapted analogue of the $2$-Wasserstein distance. To define this distance, we first recall the notion of a bicausal coupling in this context.
Here we denote by $(\omega,\bar\omega)$ a general element in $\Omega\times\Omega$.

\begin{definition}[Causal and bicausal coupling]
A probability measure $\gamma$ on $\Omega\times\Omega$ is called a \emph{causal} coupling between $\PP_1$ and $\PP_2$ if it has first marginal $\PP_1$, second marginal $\PP_2$, and satisfies
\[
\forall\ t\in[0,1], \forall\ A\in\Fcal_t,\qquad \omega \mapsto \gamma_\omega(A) \quad \text{ is $\Fcal_t$-measurable},
\]
where $(\Fcal_t)_{t\in[0,1]}$ is the $\PP_1$-completed canonical filtration,
and
$\gamma_\omega$ a  regular conditional distribution of $\gamma$ w.r.t. the first marginal. We denote by $\Pi_c(\PP_1,\PP_2)$  the set of causal couplings. 
If, in addition, $e_\#\gamma \in \Pi_c(\PP_2,\PP_1)$, where $e(\omega,\bar\omega):=(\bar\omega,\omega)$, then $\gamma$ is called \emph{bicausal}. We denote the set of such couplings by $\Pi_{bc}(\PP_1,\PP_2)$.
\end{definition}
Intuitively, causality means that, at any time $t$, the evolution of the second coordinate up to time $t$ may depend on the evolution of the first coordinate only up to time $t$, and not through its future trajectory. Bicausality imposes the same non-anticipation property in the other direction as well.
We can now define the adapted $2$-Wasserstein distance on the class of square-integrable distributions on $\Omega$, as an optimal transport problem restricted to bicausal couplings. 
For $\PP_1,\PP_2\in\Prob_2(\Omega)$, we set 
\begin{equation}\label{eq.AW}
\Acal\Wcal_2(\PP_1,\PP_2) :=
\inf_{\gamma\in\Pi_{bc}(\PP_1,\PP_2)}
\left(\EE^\gamma\left[[\omega-\bar\omega]_1\right]\right)^{\frac12},
\end{equation}
where $[\hspace*{0.03cm}\cdot\hspace*{0.03cm}]_1$ denotes the pathwise quadratic variation at time $1$.
Note that, if $\PP_1,\PP_2$ are martingale laws, then, under any bicausal coupling $\gamma\in\Pi_{bc}(\PP_1,\PP_2)$, the canonical
process on $\Omega\times\Omega$ is a $\gamma$-martingale in its own filtration, and so is the difference of the two coordinates; 
see \cite[Lemma~6.2]{BaBeHuKa20}.

Denote by $\Mart_D(\mu,\nu)$ the set of probability measures on $\Omega$ under which the canonical process is a square-integrable martingale with initial law $\mu$ and terminal law $\nu$, and set $\PP^X:=\Law(X)$. We consider the projection problem  
\begin{equation}
\label{eq:aw-projection}
    \inf_{\PP^M\in\Mart_D(\mu,\nu)} \Acal\Wcal_2(\PP^X,\PP^M),
\end{equation}
and show that the law of an $X$-Bass martingale provides one solution to it. Specifically, in analogy to \cite[Proposition~6.4]{BaBeHuKa20},
we obtain the following result, that shows the relationship between \eqref{eq:aw-projection} and \eqref{eq:dyn-OT-1}.

\begin{proposition}
Let $\mu, \nu \in \Prob_2(\RR)$ satisfy $\mu \preceq_c \nu$, $X$ be a square-integrable additive martingale, and set $q:= \Law(X_1)$. Then
\begin{equation}\label{eq.AWineq}
\inf_{\PP^M \in \Mart_D(\mu,\nu)} \Acal\Wcal_2^2(\PP^X,\PP^M) \geq \inf_{M\in\mathfrak M_X(\mu,\nu)}\Ical_X(M).
\end{equation}
Moreover, if there exists an $X$-Bass martingale $M^*$ such that the coupling $\Law(X,M^*)$ is bicausal, then equality holds in \eqref{eq.AWineq}, and $\Law(M^*)$ is an optimizer of the problem on the LHS.
\end{proposition}
\begin{proof}
Fix $\PP^M \in \Mart_D(\mu,\nu)$ and $\gamma \in \Pi_{bc}(\PP^X,\PP^M)$. From the comment after \eqref{eq.AW}, $((\omega,\bar\omega)\mapsto\omega-\bar\omega)_\#\gamma$ is a martingale law. Therefore, since $\omega_0=0$,
\[
\EE^\gamma\left[[\omega-\bar\omega]_1\right]
= \EE^\gamma\left[|\omega_1-\bar\omega_1|^2\right]-\int_\RR |x|^2\,\mu(dx).
\]
Now let $\pi:=\Law^\gamma(\bar\omega_0,\bar\omega_1)\in\Mart(\mu,\nu)$ and disintegrate $\pi(dx,dy)=\mu(dx)\pi_x(dy)$. By causality and the fact that $\omega_0=0$, $\bar\omega_0$ is independent of $\omega$, so conditionally on $\bar\omega_0=x$ the law of $(\bar\omega_1,\omega_1)$ is a coupling of $\pi_x$ and $q$. Therefore,
\[
\EE^\gamma\left[[\omega-\bar\omega]_1\right]
\geq \int_\RR\Wcal_2^2(\pi_x,q)\,\mu(dx)-\int_\RR|x|^2\,\mu(dx)
\geq \inf_{M\in\mathfrak M_X(\mu,\nu)}\Ical_X(M),
\]
where the last inequality follows from  \eqref{def:WT^q_I} and Proposition~\ref{prop:dynamic-optimality-X-Bass}. Taking the infimum over $\gamma\in\Pi_{bc}(\PP^X,\PP^M)$ proves the first claim.

Now suppose that an $X$-Bass martingale $M^*$ exists and that $\gamma^X=\Law(X,M^*)$ is bicausal. By Proposition~\ref{prop:dynamic-optimality-X-Bass}, $M^*$ minimizes $\Ical_X$, while by the definition of $\Ical_X$, $EE^{\gamma^X}\left[[X-M^*]_1\right]=\Ical_X(M^*)$. Hence
\[
\Acal\Wcal_2^2(\PP^X,\Law(M^*))\leq \Ical_X(M^*)
=\inf_{M\in\mathfrak M_X(\mu,\nu)}\Ical_X(M).
\]
Together with the first part, this yields equality and proves that $\Law(M^*)$ is an adapted Wasserstein projection of $\PP^X$ onto $\Mart_D(\mu,\nu)$.
\end{proof}
Note that, if $M^*$ is an $X$-Bass martingale from $\mu$ to $\nu$, the coupling $\gamma^X:=\Law(X,M^*)$
need not be bicausal. Indeed, casuality is immediate in the direction from $X$ to $M^*$, whereas the reverse direction requires additional care. If, for some $t\in(0,1)$, the map $f_t$ in the representation
$M_t^*=f_t(A+X_t)$
fails to be injective $\Law(X_t)$-a.e., then 
the argument used to establish reverse causality in the Brownian setting in \cite{BaBeHuKa20} breaks down. This issue can be avoided under conditions ensuring the injectivity of $f_t$ for every $t<1$, for instance when the increment laws of $X$ are equivalent to Lebesgue measure on $\RR$, as it is in the Brownian setting. Under this assumption, the bicausality of $\gamma^X$ can be established by the same argument as in \cite[Lemma~6.3]{BaBeHuKa20}.

\appendix
\section{Appendix}

\begin{proposition}[Approximation of the initial marginal]
\label{prop:approx_mu} 
Let $\mu, \nu \in \Prob_1(\RR)$ be such that  $(\mu,\nu)$ is irreducible. Then there exists a sequence $(\mu_k)_{k \in \NN}\subseteq \Prob_1(\RR)$ such that, for every $k \in \NN$,
\begin{enumerate}[label=(\roman*)]
	\item $\mu_k$ is $k$-atomic with equal weights;
	\item $(\mu_k,\nu)$ is irreducible;
	\item $U_{\mu_k} \to U_\mu$ pointwise on $(0,1)$.
\end{enumerate}
\end{proposition}

\begin{proof}
Fix $k \in \NN$, and consider the polygonal chain $\Ccal^{(k)}$ whose vertices are
\[
\left(\frac{i}{k},\, U_\mu\left(\frac{i}{k}\right)\right),
\qquad i=0,\dots,k.
\]
Since $U_\mu$ is convex, the polygonal chain $\Ccal^{(k)}$ is convex. Therefore, by \cite[Proposition 1.8]{AcMa26_semidiscrete}, $\Ccal^{(k)}$ is the graph of the integrated quantile function of a $k$-atomic probability measure $\mu_k$ with equal weights. The same proposition also implies that
 $(\mu_k,\nu)$ is irreducible.
By construction, the function $U_{\mu_k}$ coincides with the piecewise affine interpolation of $U_\mu$ on the grid points $i/k$, $i=0,\dots,k$. Since integrated quantile functions are continuous, it follows that
\[
U_{\mu_k}(p)\to U_\mu(p)
\qquad \text{for every } p\in (0,1).
\]
\end{proof}

\begin{corollary}[Approximation of the terminal marginal]\label{cor:approx_nu_2}
Let $\nu\in \Prob_1(\RR)$. Then there exists a sequence $(\nu_k)_{k \in \NN}\subseteq \Prob_1(\RR)$ such that, for all $k \in \NN$,
\begin{enumerate}[label=(\roman*)]
\item\label{cor:approx_nu_2:ac} $\nu_k \ll \lambda$ and $\supp(\nu_k)$ is a bounded interval;
\item\label{cor:approx_nu_2:lb} $\rho_{\nu_k}$ is bounded away from $0$ on $\supp(\nu_k)$;
\item\label{cor:approx_nu_2:pt} $Q_{\nu_k}(p)\to Q_\nu(p)$ pointwise on $(0,1)$;
\item\label{cor:approx_nu_2:dom} there exists $Q\in L^1(0,1)$ such that $|Q_{\nu_k}(p)|\le Q(p)$, for all $p\in(0,1)$.
\end{enumerate}
\end{corollary}
\begin{proof}
Set $m_\nu = \mean(\nu)$. If $\nu = \delta_{m_\nu}$, take $\nu_k := \text{Unif}_{[m_\nu-1/k, m_\nu+1/k]}$. Otherwise, let $\mu := \delta_{m_\nu}$. Then $\mu$ is $1$-atomic, and the pair $(\mu,\nu)$ is irreducible. Hence \cite[Lemma A.1]{AcMa26_semidiscrete} applies to $(\mu,\nu)$ and yields a sequence $(\nu_k)_{k \in \NN} \subseteq \Prob_1(\RR)$ such that \ref{cor:approx_nu_2:ac}--\ref{cor:approx_nu_2:dom} hold.
\end{proof}

\begin{proposition}[Approximation of the reference measure II]\label{prop:approx_q_2}
Let $q\in \Prob(\RR)$ be such that $q \ll \lambda$.
Then there exists $(q_k)_{k\in\NN} \subseteq \Prob(\RR)$ such that, for every $k \in \NN$,
\begin{enumerate}[label=(\roman*)]
    \item $\supp(q_k)=\RR$;
    \item $q_k \ll \lambda$ and $\rho_{q_k} \in L^\infty(\RR)$;
    \item $\|\rho_{q_k}-\rho_q\|_{L^1(\RR)} \to 0$ as $k\to\infty$.
\end{enumerate}
\end{proposition}

\begin{proof}
By \cite[Proposition A.2]{AcMa26_semidiscrete}, there exist $(\widetilde q_k)_{k\in\NN}\subseteq\Prob(\RR)$ with bounded densities such that $\rho_{\widetilde q_k}\to\rho_q$ in $L^1(\RR)$.
Let $\gamma_{1/k}$ be the centered Gaussian law with variance $1/k$, and set $q_k:=\widetilde q_k*\gamma_{1/k}$.
Then $q_k\ll\lambda$, with bounded density $\rho_{q_k}=\rho_{\widetilde q_k}*\gamma_{1/k}$, and $\supp(q_k)=\RR$.
Moreover,
\begin{align*}
	\|\rho_{q_k}-\rho_q\|_{L^1} & \leq \|(\rho_{\widetilde q_k}-\rho_q) \ast \gamma_{1/k}\|_{L^1} + \|\rho_q*\gamma_{1/k}-\rho_q\|_{L^1}\\
	& \leq \|\rho_{\widetilde q_k}-\rho_q\|_{L^1}\|\rho_{\gamma_{1/k}}\|_{L^1} + \|\rho_q*\gamma_{1/k}-\rho_q\|_{L^1} \to 0.
\end{align*}
\end{proof}

\begin{proposition}[Helly's selection theorem for possibly unbounded monotone functions]
\label{prop:helly-selection-thm}
Let $(f_n)_{n \in \NN} \subseteq \Mon((0,1), \RR)$. Then there exist a subsequence $(f_{n_k})_{k \in \NN}$, $f \in \Mon((0,1), \overline \RR)$, and two points $u_1^*,u_2^*\in[0,1]$ with $u_1^*\le u_2^*$, such that
\[
f_{n_k}(u)\longrightarrow f(u), \qquad \text{for a.e. }u\in(0,1),
\]
and
\[
f(u)=-\infty \quad \text{if }u<u_1^*, \qquad f(u)\in\RR \quad \text{if }u_1^*<u<u_2^*, \qquad f(u)=+\infty \quad \text{if }u>u_2^*.
\]
\end{proposition}

\begin{proof}
Define
\[
g_n(u):=\arctan(f_n(u)), \qquad u\in(0,1).
\]
Then each $g_n$ is non-decreasing and takes values in
$\left(-\frac{\pi}{2},\frac{\pi}{2}\right)$,
so the sequence $(g_n)_{n\ge1}$ is uniformly bounded. By Helly's selection theorem, there exist a subsequence $(g_{n_k})_{k\ge1}$ and a non-decreasing function $\widetilde g:(0,1)\to[-\frac{\pi}{2},\frac{\pi}{2}]$ such that
\[
g_{n_k}(u)\longrightarrow\widetilde g(u)
\]
at every continuity point of $\widetilde g$. Since a monotone function has at most countably many discontinuities, the convergence holds for a.e. $u\in(0,1)$.
Define the left-continuous modification of $\widetilde g$ by
\[
g(u):=\widetilde g(u-):=\sup_{v<u}\widetilde g(v), \qquad u\in(0,1).
\]
Then $g$ is non-decreasing and left-continuous on $(0,1)$. Moreover, $g(u)=\widetilde g(u)$ at every continuity point of $\widetilde g$, and hence
\[
g_{n_k}(u)\longrightarrow g(u), \qquad \text{for a.e. }u\in(0,1).
\]

Now define $f:(0,1)\to[-\infty,+\infty]$ by
\[
f(u)=\begin{cases}-\infty, & g(u)=-\frac{\pi}{2},\\ \tan(g(u)), & g(u)\in(-\frac{\pi}{2},\frac{\pi}{2}),\\ +\infty, & g(u)=\frac{\pi}{2}.\end{cases}
\]
Since $g$ is non-decreasing, so is $f$. As the extended tangent map is continuous and increasing from $[-\frac{\pi}{2},\frac{\pi}{2}]$ to $\overline\RR$, the left-continuity of $g$ implies that $f$ is left-continuous on $(0,1)$. Moreover,
\[
f_{n_k}(u)=\tan(g_{n_k}(u))\longrightarrow f(u), \qquad \text{for a.e. }u\in(0,1),
\]
where the convergence is understood in the extended real line.

Set
\[
A:=\{u\in(0,1):f(u)=-\infty\}, \qquad B:=\{u\in(0,1):f(u)=+\infty\}.
\]
Since $f$ is non-decreasing, $A$ is either empty or an initial interval, while $B$ is either empty or a terminal interval. Define
\[
u_1^*:=\begin{cases}\sup A, & A\neq\emptyset,\\ 0, & A=\emptyset,\end{cases} \qquad u_2^*:=\begin{cases}\inf B, & B\neq\emptyset,\\ 1, & B=\emptyset.\end{cases}
\]
Clearly $u_1^*\le u_2^*$. By construction,
\[
f(u)=-\infty \quad \text{for every }u<u_1^*, \qquad f(u)=+\infty \quad \text{for every }u>u_2^*.
\]
Finally, if $u_1^*<u<u_2^*$, then $u\notin A\cup B$, hence
$f(u)\in\RR$.
\end{proof}

\printbibliography

\end{document}